\documentclass[10pt]{amsart}

\usepackage{tikz}
\usepackage{graphicx}
\usepackage{tikz-cd}

\newtheorem{theorem}{Theorem}[section]
\newtheorem{lemma}[theorem]{Lemma}
\newtheorem{proposition}[theorem]{Proposition}
\newtheorem{corollary}[theorem]{Corollary}

\theoremstyle{definition}
\newtheorem{definition}[theorem]{Definition}

\theoremstyle{remark}
\newtheorem{remark}[theorem]{Remark}

\numberwithin{equation}{section}

\newcommand{\abs}[1]{\lvert#1\rvert}

\begin{document}

\title{$A_{n}$-type surface singularity and nondisplaceable Lagrangian tori}

\author{Yuhan Sun}
\address{Department of Mathematics, Stony Brook University, Stony Brook, New York, 11794}

\email{sun.yuhan@stonybrook.edu}

\begin{abstract}
We prove the existence of a one-parameter family of nondisplaceable Lagrangian tori near a linear chain of Lagrangian 2-spheres in a symplectic 4-manifold. When the symplectic structure is rational we prove that the deformed Floer cohomology groups of these tori are nontrivial. The proof uses the idea of toric degeneration to analyze the full potential functions with bulk deformations of these tori.
\end{abstract}

\maketitle

%% The correct journal style for \specialsection is all uppercase; a known bug
%% in amsart.cls prevents this, so input must be uppercase until it is fixed.

\tableofcontents

\section{Introduction}
A compact submanifold $L$ of a connected symplectic manifold $\left(X, \omega\right)$ is called nondisplaceable if it cannot be separated from itself by any Hamiltonian diffeomorphism. That is,
$$
L\cap \phi\left(L\right) \neq \emptyset, \quad \forall \phi\in Ham\left(X, \omega\right).
$$
When $L$ is a Lagrangian submanifold the nondisplaceability of $L$ is related with the Lagrangian intersection Floer theory. In particular when $X$ is a symplectic toric manifold/orbifold and $L$ is a smooth torus orbit, the problem of nondisplaceability has been studied in \cite{ABM}, \cite{B}, \cite{CO}, \cite{CP}, \cite{FOOO1, FOOO2, FOOO3}, \cite{WW}, \cite{W}, \cite{WU} by either combinatorial or geometrical methods. One of these approaches is that Fukaya-Oh-Ohta-Ono established a critical point theory to reduce the calculation of the Floer cohomology of a toric fiber to its potential function with bulk deformations. Later in \cite{FOOO3} they combined this theory with the idea of toric degeneration on a Hirzebruch surface $F_{2}$ to get a one-parameter family of nondisplaceable Lagrangian tori in $S^{2}\times S^{2}$, which are not toric fibers. In this note we generalize the similar idea to find nondisplaceable Lagrangian tori in a closed symplectic 4-manifold which contains a chain of Lagrangian 2-spheres.

\subsection{Main results}
Let $M_{n, \epsilon, R}$ be the $A_{n}$-type Milnor fiber, that is,
$$
M_{n, \epsilon, R}=\lbrace z=(z_{1}, z_{2}, z_{3})\mid z_{1}^{2}+z_{2}^{2}+z_{3}^{n+1}=\epsilon, \abs{z}\leq R \rbrace\subset \mathbb{C}^{3}
$$
where $\epsilon$ is a complex parameter and $R>0$ is a real parameter. We omit $R$ in the notation when there is no confusion. It is known that there is a one-parameter family of Lagrangian tori $L_{n, \lambda}$ with nontrivial Floer cohomology in $M_{n, \epsilon}$, which we will describe in section 4. All of these Lagrangian tori are monotone with minimal Maslov number 2 and monotonicity constant $\lambda$, where $\lambda$ is parameterized by an open interval in $\mathbb{R}$.

Symplectically the Milnor fiber $M_{n, \epsilon}$ can be regarded as a linear plumbing of the disk cotangent bundles of 2-spheres, with a rearrangement on the boundary. Let $X$ be a symplectic 4-manifold which contains a chain of Lagrangian 2-spheres. Then there exists a small neighborhood of this chain which is symplectomorphic to certain $M_{n, \epsilon}$. Our main theorem can be stated as follows.

\begin{theorem}
Let $X$ be a closed symplectic 4-manifold which contains a linear chain of Lagrangian 2-spheres. Consider the Lagrangian embedding
$$
L_{n, \lambda}\hookrightarrow M_{n, \epsilon}\subset X
$$
then $L_{n, \lambda}$ is nondisplaceable in $X$ for all $\lambda$ in a possibly smaller open interval.
\end{theorem}

In particular when $n=1$ the Milnor fiber is the disk cotangent bundle of a 2-sphere. Hence we have the following corollary.

\begin{corollary}
Let $X$ be a closed symplectic 4-manifold which contains a Lagrangian 2-sphere. Then there is a one-parameter family of nondisplaceable Lagrangian tori in $X$ near this Lagrangian sphere.
\end{corollary}

Note that for $n<n'$ we have an inclusion from $M_{n, \epsilon}$ to $M_{n', \epsilon}$ after slightly shrinking the domain so we can compare those tori $L_{n, \lambda}$ and $L_{n', \lambda}$. We remark that they are different as they bound different numbers of holomorphic disks($L_{n, \lambda}$ bounds $2^{n+1}$ classes of holomorphic disks). Hence there may exist lots of different families of nondisplaceable Lagrangian tori in $X$, depending on the length of the chain of Lagrangian 2-spheres. On the question how those different tori intersect with each other, we suggest the work of Tonkonog-Vianna in \cite{TV}.

In Theorem 1.1 we do not require that the chain is maximal. Hence it implies the existence of families of nondisplaceable Lagrangian tori in many interesting settings. For example, the smoothing of a symplectic orbifold with $ADE$-type singularities which contains $ADE$-chains of Lagrangian spheres, the compactification of Milnor fibers of many hypersurface singularities like $K3$ surfaces from 14 exceptional singularities.

Although locally the Floer cohomology of $L_{n, \lambda}$ is nontrivial in $M_{n, \epsilon}$ we do not know whether it is nontrivial in $X$. The proof of Theorem 1.1 is based on that the deformed Floer cohomology of $L_{n, \lambda}$ in $X$ is nontrivial modulo any large energy parameter. When $X$ is equipped with a rational symplectic structure we have the following positive answer.

\begin{theorem}
Let $X$ be a closed symplectic 4-manifold with a rational symplectic structure which contains a chain of Lagrangian 2-spheres. Consider the Lagrangian embedding
$$
L_{n, \lambda}\hookrightarrow M_{n, \epsilon}\subset X
$$
then for any $L_{n, \lambda}$ there exists some bulk deformation $\mathfrak{b}(\lambda)$ such that $L_{n, \lambda}$ has nontrivial deformed Floer cohomology with respect to $\mathfrak{b}(\lambda)$ in $X$. Here $\lambda$ is parameterized by an open interval.
\end{theorem}

One concrete application is that we find new families of Lagrangian tori with nontrivial deformed Floer cohomology in the $k$-points blow up of the complex projective plane where $k=2,3,4,5$.

\begin{theorem}
Let $X$ be a closed semi-Fano toric surface and $\hat{X}$ be the smoothing of its toric degeneration.
\begin{enumerate}
\item There exist one-parameter families of Lagrangian tori with nontrivial Floer cohomology with bulk deformations in $\hat{X}$.
\item For some chosen bulk deformation $\mathfrak{b}$ all the critical points of the potential function are Morse and the deformed Kodaira-Spencer map induces an isomorphism between the deformed quantum cohomology and the Jacobian ring of the potential function
$$
QH_{\mathfrak{b}}(\hat{X}; \Lambda) \xrightarrow{\mathfrak{ks}_{\mathfrak{b}}} Jac(\mathfrak{PO}_{\mathfrak{b}}; \Lambda).
$$
\item There exist one-parameter families of linearly independent quasi-morphisms $\mu^{u}_{\mathfrak{b}(u)}$ on the universal cover of the Hamiltonian diffeomorphism group of $\hat{X}$.
\end{enumerate}
\end{theorem}

The smoothing $\hat{X}$ is not toric and we obtain the potential function by combining the computation of Chan-Lau in \cite{CL} with a local deformation. There are 11 types of semi-Fano toric surfaces which are classified by their moment polytopes. The proof is a case by case computation.

\begin{remark}
\begin{enumerate}
\item The cases of $S^{2}\times S^{2}$ and the monotone cubic surface have been done in \cite{FOOO3} and \cite{FOOO5}. Also the case of $X_{5}$ has been worked out by Vianna in \cite{V3}.
\item In \cite{FOOO2} a one-parameter family of nondisplaceable Lagrangian tori has been found in each of the k-points blow up of $\mathbb{C}P^{2}$ with $k\geq 2$. When $k=2$ our family of tori is different from that in \cite{FOOO2} since our tori bound 6 families of holomorphic disks but those in \cite{FOOO2} bound 5 families. When $k\geq 3$ the k-points blow up of $\mathbb{C}P^{2}$ in \cite{FOOO2} is not monotone and we mostly focus on the monotone case.
\end{enumerate}
\end{remark}

The outline of this note is the following. In section 2 we give brief preliminaries on background and in section 3 we use the example of a semi-Fano toric surface to prove the deformation invariance of one-point open Gromov-Witten invariants under local deformation. In section 4 we prove the existence of critical points of the potential function of $L_{n, \lambda}$ in a general symplectic 4-manifold $X$, which completes Theorem 1.1 and Theorem 1.3. In the appendix we provide a useful algebraic lemma and list explicit results about the existence and nondegeneracy of critical points in the degeneration of other semi-Fano toric surfaces.

\subsection*{Acknowledgements}
The author acknowledges his advisor Kenji Fukaya for suggesting this project and frequent enlightening guidance during these years, as well as the financial support in 2017-2018 academic year. The author acknowledges Gao Chen, Yuan Gao and Mark McLean for many helpful discussions on various topics. The author acknowledges Kwokwai Chan, Siu-Cheong Lau, Dmitry Tonkonog, Renato Vianna and Weiwei Wu for kindly explaining their related work. Especially the author acknowledges Jason Starr for explanations on many algebraic lemmas.

\section{Preliminaries}
We give a very brief summary to the theory of deformed Floer cohomology and potential functions, referring to Section 2 and Appendix 1 in \cite{FOOO3} for more details.

First we specify the ring and field that will be used. The Novikov ring $\Lambda_{0}$ and its field $\Lambda$ of fractions are defined by
$$
\Lambda_{0}=\lbrace \sum_{i=0}^{\infty}a_{i}T^{\lambda_{i}}\mid a_{i}\in \mathbb{C}, \lambda_{i}\in\mathbb{R}_{\geq 0}, \lambda_{i}<\lambda_{i+1}, \lim_{i\rightarrow \infty}\lambda_{i}=+\infty \rbrace
$$
and
$$
\Lambda=\lbrace \sum_{i=0}^{\infty}a_{i}T^{\lambda_{i}}\mid a_{i}\in \mathbb{C}, \lambda_{i}\in\mathbb{R}, \lambda_{i}<\lambda_{i+1}, \lim_{i\rightarrow \infty}\lambda_{i}=+\infty \rbrace
$$
where $T$ is a formal variable. The maximal ideal of $\Lambda_{0}$ is defined by
$$
\Lambda_{+}=\lbrace \sum_{i=0}^{\infty}a_{i}T^{\lambda_{i}}\mid a_{i}\in \mathbb{C}, \lambda_{i}\in\mathbb{R}_{>0}, \lambda_{i}<\lambda_{i+1}, \lim_{i\rightarrow \infty}\lambda_{i}=+\infty \rbrace.
$$
We remark that the field $\Lambda$ is algebraically closed since the ground field is $\mathbb{C}$, see Appendix A in \cite{FOOO1}. All the nonzero elements in $\Lambda_{0}-\Lambda_{+}$ are units in $\Lambda_{0}$. Next we define a valuation $v$ on $\Lambda$ by
$$
v(\sum_{i=0}^{\infty}a_{i}T^{\lambda_{i}})=\inf \lbrace \lambda_{i}\mid a_{i}\neq 0\rbrace, \quad v(0)=+\infty.
$$
This valuation gives us a non-Archimedean norm
$$
\abs{a=\sum_{i=0}^{\infty}a_{i}T^{\lambda_{i}}}=e^{-v(a)}.
$$

Let $X$ be a smooth symplectic 4-manifold and $L$ be a Lagrangian torus in $X$. Fukaya-Oh-Ohta-Ono defined the deformed Floer cohomology
$$
HF(L; \mathfrak{b}, b), \quad \mathfrak{b}\in H^{2}(X; \Lambda_{0}), \quad b\in \widehat{\mathcal{M}}_{weak}\left(L, \mathfrak{m}^{\mathfrak{b}}\right)\subset H^{1}\left(L; \Lambda_{0}\right)
$$
where $\mathfrak{b}$ is a bulk deformation and $b$ is a weak bounding cochain, see Section 2 in \cite{FOOO3} for precise definitions. To compute this deformed Floer cohomology we introduce the following notion of a potential function of a Lagrangian torus.

Let $\beta\in H_{2}(X, L; \mathbb{Z})$ be a relative homology class of Maslov index 2. Let $\mathcal{M}_{1}(\beta, J)$ be the compactified moduli space of $J$-holomorphic disks of class $\beta$ with boundary on $L$ and one boundary marked point. When $J$ is generic, it has been shown that $\mathcal{M}_{1}(\beta, J)$ carries a rational fundamental cycle of real dimension 2, see Condition 6.1 in\cite{FOOO3}. The one-point open Gromov-Witten invariant is defined as the mapping degree
$$
n_{\beta}= \deg(ev: \mathcal{M}_{1}(\beta, J) \rightarrow L)
$$
which is a rational number. If the Lagrangian $L$ is monotone with minimal Maslov number 2 then $n_{\beta}$ is an invariant of the choice of $J$. If $L$ is not monotone then $n_{\beta}$ depends on the choice of a generic $J$.

The potential function of a Lagrangian torus can be defined as follows
\begin{equation}
\mathfrak{PO}_{\mathfrak{b}}^{L}(b)=\sum_{\mu(\beta)=2} n_{\beta}e^{\mathfrak{b}\cap \beta} e^{b\cap \partial \beta} T^{\omega(\beta)}
\end{equation}
where
$$
\mathfrak{b}\in H^{2}(X; \Lambda_{0}), \quad b\in \widehat{\mathcal{M}}_{weak}\left(L, \mathfrak{m}^{\mathfrak{b}}\right)
$$
see Section 7 and Appendix 1 in \cite{FOOO3} for computations in detail. Actually since our Lagrangian is a 2-torus, for a generic almost complex structure we can assume that
$$
H^{1}\left(L; \Lambda_{0}\right)\big/ H^{1}\left(L; 2\pi \sqrt{-1}\mathbb{Z}\right) \subset
\widehat{\mathcal{M}}_{weak}\left(L, \mathfrak{m}^{\mathfrak{b}}\right)
$$
by a degree counting argument like Remark A.2 in \cite{FOOO3}. For a fixed bulk deformation $\mathfrak{b}$, we choose a basis $e_{1}, e_{2}$ of $H^{1}(L; \mathbb{Z})$ and write $b= x_{1}e_{1}+x_{2}e_{2}$ for $x_{1}, x_{2}\in \Lambda_{0}$. Then the potential function can be regarded as certain kinds of Laurent series in terms of $y_{1}=e^{x_{1}}, y_{2}=e^{x_{2}}$ with coefficients in $\Lambda_{+}$. Then we have the following theorem on the critical point theory of the potential function.
\begin{theorem}
(Theorem 2.3, \cite{FOOO3}) Let $L$ be a Lagrangian torus in $\left(X, \omega\right)$. Suppose that
$$
H^{1}\left(L; \Lambda_{0}\right)\big/ H^{1}\left(L; 2\pi \sqrt{-1}\mathbb{Z}\right) \subset
\widehat{\mathcal{M}}_{weak}\left(L, \mathfrak{m}^{\mathfrak{b}}\right)
$$
and $b\in H^{1}\left(L; \Lambda_{0}\right)$ is a critical point of the potential function $\mathfrak{PO}_{\mathfrak{b}}^{L}$. Then we have
$$
HF\left(\left(L; \mathfrak{b}, b\right), \left(L; \mathfrak{b}, b\right)\right)\cong H\left(L\right).
$$
In particular $L$ is nondisplaceable.
\end{theorem}
As we remarked above, the first condition is satisfied when $L$ is a 2-torus and the almost complex structure is generic. With this understood we directly use the critical point theory to locate nondisplaceable Lagrangian tori.

\section{Examples: del Pezzo surfaces}
In this section we work on an explicit example to show how to find nondisplaceable Lagrangian tori by the method of toric degeneration with a local deformation.

Let $(X, \omega, J)$ be a symplectic toric surface where the symplectic and complex structures are induced from a moment polytope. We say $X$ is Fano if all holomorphic spheres in $X$ have positive Chern numbers and it is semi-Fano if all holomorphic spheres have nonnegative Chern numbers. Semi-Fano toric surfaces are classified by their moment polytopes and there are 16 of them in total, 5 of which are Fano. In the rest of this note we use the name ``semi-Fano'' to represent those 11 surfaces which are not Fano. If we collapse those spheres with zero Chern number to points in a toric way we get a toric symplectic orbifold with $A_{n}$-type singularities. Then we can perform the smoothing operation by cutting and gluing appropriate Milnor fibers to get a smooth symplectic manifold $\hat{X}$. In \cite{CL} Chan and Lau carefully studied the one-point open Gromov-Witten invariants of toric fibers in $X$ to compute the disk potential function. Combining their computation with a deformation argument we compute the disk potential function with respect to the smoothing symplectic structure on $\hat{X}$.  We label the above 11 surfaces as $X_{1}, \cdots, X_{11}$ exactly following the notation in the appendix of \cite{CL}.

Let $L(u)$ be a toric fiber of a semi-Fano toric surface $X$. When the energy of -2-curves are very small we can deform the symplectic and complex structures in a neighborhood of those curves such that they become Lagrangian spheres, see Figure 1. This deformation happens away from a neighborhood of $L(u)$, which is represented by an interior dot in Figure 1. We will show that all the one-point open Gromov-Witten invariants are unchanged during the deformation.

\begin{figure}
  \begin{tikzpicture}[xscale=0.7, yscale=0.7]
  \path [fill=lightgray, lightgray] (0,-2)--(0,0)--(-2,-3);
  \draw (0,-2)--(0,0)--(-2,-3);
  \filldraw[black] (-0.5,-1.8) circle (1pt);
  \draw [->] (1,-1) -- (4,-1);
  \draw [->] (-2,-1) -- (-5,-1);
  \draw (-6,-2)--(-6,-1)--(-6.5,-1)--(-7,-1.5)--(-8,-3);
  \path [fill=lightgray, lightgray] (-6,-2)--(-6,-1)--(-6.5,-1)--(-7,-1.5)--(-8,-3);
  \filldraw[black] (-6.5,-1.8) circle (1pt);
  \draw (6,-1.5) .. controls (6.2,-1.5) and (6.5,0) .. (7,-1);
  \draw (7,-2)--(7,-1);
  \draw (5,-3)--(6,-1.5);
  \node [above] at (-3.5,-1) {toric blow up};
  \node [above] at (2.5,-1) {smoothing};
  \node [below] at (0,-3) {orbifold};
  \filldraw[black] (-6,-1) circle (1pt);
  \filldraw[black] (-6.5,-1) circle (1pt);
  \filldraw[black] (-7,-1.5) circle (1pt);
  \filldraw[black] (7,-1) circle (1pt);
  \filldraw[black] (6,-1.5) circle (1pt);
  \filldraw[black] (0,0) circle (1pt);
  \path [fill=lightgray, lightgray] (5,-3)--(6,-1.5) .. controls (6.2,-1.5) and (6.5,0) .. (7,-1)--(7,-2);
  \filldraw[black] (6.5,-1.8) circle (1pt);
  \end{tikzpicture}
  \caption{Toric resolution and smoothing.}
  \label{fig: Toric resolution and smoothing}
\end{figure}
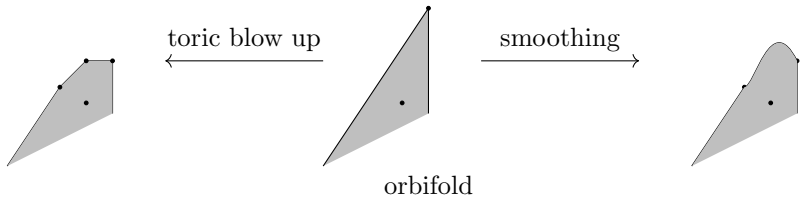

\subsection{Resolution and smoothing of the $A_{n}$-type singularity}
The cyclic group $\mathbb{Z}_{n+1}\cong \lbrace \eta \in \mathbb{C} \mid \eta^{n+1}=1 \rbrace$ acts on $\mathbb{C}^{2}$ by
$$
\eta\cdot (z_{1}, z_{2})=(\eta^{-1}z_{1}, \eta z_{2})
$$
which has a fixed point at the origin. We call this quotient singularity an $A_{n}$-type singularity. It is a toric singularity and admits a toric minimal resolution, the Hirzebruch-Jung resolution. Geometrically it is obtained by consecutively blowing up the singular locus $n$ times and in the moment polytope it is done by chopping corners with certain slopes. For details about the toric minimal resolution of $A_{n}$-type singularity, see section 4.3 in \cite{ABM}.

Note that the $A_{n}$-type singularity is also a hypersurface singularity which is isomorphic to
$$
M_{n, 0, +\infty}=\lbrace z=(z_{1}, z_{2}, z_{3})\mid z_{1}^{2}+z_{2}^{2}+z_{3}^{n+1}=0 \rbrace\subset \mathbb{C}^{3}
$$
near the origin. Therefore we can perform the smoothing operation where a neighborhood of the origin is replaced by a proper Milnor fiber. The Milnor fiber $M_{n, \epsilon, R}$ is a smooth manifold with boundary and inherits an exact symplectic structure from $\mathbb{C}^{3}$. The boundary of $M_{n, \epsilon, R}$ topologically is a lens space. After rearranging the boundary we can make it the lens space with standard contact structure. Next we glue $M_{n, \epsilon, R}$ with $M_{n, 0, +\infty}-M_{n, \epsilon, R'}$ by using the standard Liouville vector field for proper $R, R'$. Since the lens space has trivial second rational cohomology group the glued symplectic form is unique up to a symplectomorphism. We call the resulting smooth symplectic manifold the smoothing of $M_{n, 0, +\infty}$.

From singularity theory, the toric minimal resolution is known to be diffeomorphic to the smoothing. But they have different symplectic and complex structures. In the resolution we have a chain of -2-curves which topologically corresponds a chain of Lagrangian spheres in the smoothing. By Kronheimer's work we know that these two symplectic structures are not isolated but connected through a family of hyperK\"ahler metrics.

\subsection{HyperK\"ahler ALE spaces}Let $D_{1}, \cdots, D_{k}$ be a maximal chain of -2 toric divisors in a semi-Fano toric surface $X$. Then after a $SL(2; \mathbb{Z})$ transformation a neighborhood of this chain in the moment polytope is isomorphic to the toric minimal resolution of $A_{n}$-type singularity. In \cite{K1} and \cite{K2} Kronheimer showed that the minimal resolution of $A_{n}$-type singularity carries a family of hyperK\"ahler metrics. We briefly summary what we need from \cite{K1} and \cite{K2} in the following and also suggest \cite{R} for a symplectic point of view.

\begin{definition}
Let $\left(M, g\right)$ be a Riemannian manifold. It is called a hyperK$\ddot{a}$hler manifold if it carries three almost complex structures $I, J, K$ satisfying the quaternion relation $IJK=-1$ and $I, J, K$ are orthogonal covariant constant with respect to the Levi-Civita connection.
\end{definition}

Therefore a hyperK$\ddot{a}$hler manifold is K$\ddot{a}$hler with respect to each of the complex structures $I, J, K$, with corresponding K$\ddot{a}$hler forms
$$
\omega_{I}=g\left(I\cdot, \cdot\right),\quad \omega_{J}=g\left(J\cdot, \cdot\right),\quad \omega_{K}=g\left(K\cdot, \cdot\right).
$$

Moreover we have an $S^{2}$-family of K$\ddot{a}$hler forms for any hyperK$\ddot{a}$hler manifold. Given a vector $u=\left(u_{I}, u_{J}, u_{K}\right)\in S^{2}\subseteq \mathbb{R}^{3}$ then we get a complex structure $I_{u}=u_{I}I+u_{J}J+u_{K}K$ and a K$\ddot{a}$hler form
$$
\omega_{u}=u_{I}\omega_{I}+u_{J}\omega_{J}+u_{K}\omega_{K}.
$$

\begin{definition}
An ALE space (asymptotically locally Euclidean) is a hyperK$\ddot{a}$hler 4-manifold with precisely one end at infinity that is isometric to $\mathbb{C}^{2}\big/G$ for a finite subgroup $G\subseteq SU(2)$. The metric on $\mathbb{C}^{2}\big/G$ differs from the flat quotient metric by order $O\left(r^{-4}\right)$ terms and has an appropriate decay in the derivatives.
\end{definition}

The following existence and uniqueness theorems for hyperK$\ddot{a}$hler structures are Theorem 1.1 and Theorem 1.3 in \cite{K1}.

\begin{theorem}
Let $Y$ be the underlying smooth manifold of some minimal resolution of $\mathbb{C}^{2}\big/ G$. Let three cohomology classes $\kappa_{1}, \kappa_{2}, \kappa_{3}\in H^{2}\left(Y; \mathbb{R}\right)$ be given which satisfy the nondegeneracy condition
$$
\forall\Sigma\in H_{2}\left(Y; \mathbb{Z}\right)\quad with \quad \Sigma\cdot\Sigma=-2,
$$
$$
\exists i\in \lbrace 1, 2, 3\rbrace \quad with \quad \kappa_{i}\left(\Sigma\right)\neq 0.
$$
Then there exists on $Y$ an ALE hyperK$\ddot{a}$hler structure for which the cohomology classes of the K$\ddot{a}$hler forms $[\omega_{i}]=\kappa_{i}$.
\end{theorem}

\begin{theorem}
If $Y_{1}$ and $Y_{2}$ are two ALE hyperK$\ddot{a}$hler 4-manifolds, and there is a diffeomorphism $Y_{1}\rightarrow Y_{2}$ under which the cohomology classes of the K$\ddot{a}$hler forms agree, then $Y_{1}$ and $Y_{2}$ are isometric.
\end{theorem}

Next we construct a family of symplectic forms on a semi-Fano toric surface $X$. Let $B_{r}$ be the ball of radius $r$ in $\mathbb{C}^{2}$. Let $\widetilde{\mathbb{C}^{2}\big/ \mathbb{Z}_{n+1}}$ be the toric resolution of the $A_{n}$-type singularity. We start with an open neighborhood $U$ of a maximal chain of -2 toric divisors $D_{1}, \cdots, D_{n}$ such that $U$ is symplectomorphic to $\widetilde{\mathbb{C}^{2}\big/ \mathbb{Z}_{n+1}}\cap B_{r}$. The toric symplectic and complex structures on $U$ are denoted by $(\omega_{0}, J_{0})$. Set $\kappa_{1}=[\omega_{0}], \kappa_{2}=\kappa_{3}=[0]$. We assume that $[\omega_{0}](D_{i})=\alpha_{i}$ where $\alpha_{i}$ are positive irrational numbers such that $\lbrace\alpha_{i}\rbrace_{1\leq i\leq n}$ are linearly independent over $\mathbb{Z}$. Then the nondegeneracy condition in above theorem is satisfied and we have an ALE hyperK$\ddot{a}$hler structure on $\widetilde{\mathbb{C}^{2}\big/ \mathbb{Z}_{n+1}}$, which we denote by $\omega_{I}, \omega_{J}, \omega_{K}$. By the uniqueness theorem we know that $\omega_{0}=\omega_{I}$ as symplectic forms up to a symplectomorphism. (Another way to see $\omega_{0}=\omega_{I}$ is to connect them in the K$\ddot{a}$hler cone and apply Moser's theorem.) Therefore we can connect $\omega_{I}$ and $\omega_{J}$ by a family of symplectic forms. The ALE property shows that for large $r'$ we can rearrange the ends such that outside a larger radius $r''$ all our $\omega_{t}$ is the standard quotient symplectic form and the boundary is the lens space with the standard contact structure. Then by a rescaling from $r''$ to $r$ we can glue the region inside the radius $r$ to $X-U$. Hence we get a deformation from $\omega_{I}$ to $\omega_{J}$ inside $U$ such that the symplectic structure outside $U$ is unchanged. In summary we have the following proposition.

\begin{proposition}
Let $X$ be a semi-Fano toric surface and $U$ be a neighborhood of a maximal chain of -2 toric divisors $D_{1}, \cdots, D_{n}$. Let $\omega_{0}$ be the toric symplectic and complex structures on $X$. Then there exists a one-parameter family $\lbrace\omega_{t}\rbrace_{t\in [0,1]}$ on $X$ such that
\begin{enumerate}
\item $\omega_{t}(D_{i})$ decreases to $\omega_{1}(D_{i})=0$, for all $1\leq i\leq n$. In particular this family avoids the symplectic structure $-\omega_{0}$;
\item $\lbrace\omega_{t}\rbrace_{t\in [0,1]}$ does not change outside the neighborhood $U$.
\end{enumerate}
\end{proposition}

\begin{remark}
About this deformation we make the following remarks.
\begin{enumerate}
\item We need to deform the almost complex structure as we deform the symplectic structure because there may not be a single almost complex structure which is compatible to all $\omega_{t}$. In practice how we choose the almost complex structure will be explained in next subsection.
\item We can choose $\omega_{1}$ to be the symplectic form on the $A_{n}$-Milnor fiber, with a rearranging on the boundary. In this case the deformation can be regarded as a family version of cutting and gluing Milnor fibers in previous subsection and in \cite{FOOO3}.
\item This gluing operation is just the ``pregluing'' in differential geometry because we do not need to go one step further to approximate some Calabi-Yau metric.
\item To get a similar deformation family we may use the simultaneous resolution model by Brieskorn \cite{BR} or the symplectic deflation technique by Li-Usher \cite{LU} to degenerate the -2-curves.
\end{enumerate}
\end{remark}

\subsection{Invariance of $n_{\beta}$}
From above discussion we have a family of smooth symplectic manifold $(X, \omega_{t})_{t\in [0,1]}$. We choose a family of almost complex structures and show that all $n_{\beta}$ are invariant if we put some assumptions on this deformation family.

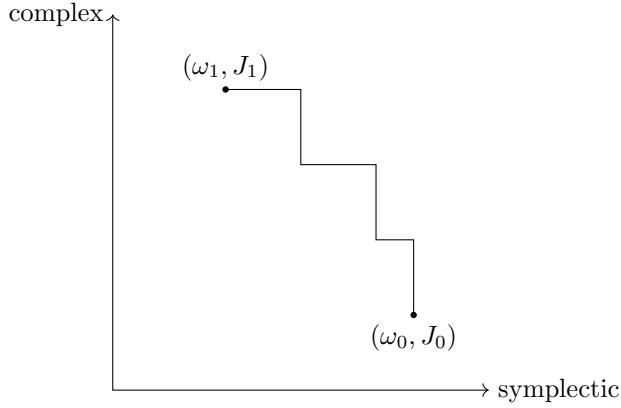
\begin{figure}
  \begin{tikzpicture}[xscale=0.5, yscale=0.5]
  \draw [<->] (0,10) -- (0,0) -- (10,0);
  \draw (8,2)--(8,4)--(7,4)--(7,6)--(5,6)--(5,8)--(3,8);
  \node [left] at (0,10) {complex};
  \node [right] at (10,0) {symplectic};
  \node [below] at (8,2) {$\left(\omega_{0}, J_{0}\right)$};
  \node [above] at (3,8) {$\left(\omega_{1}, J_{1}\right)$};
  \filldraw[black] (8,2) circle (2pt);
  \filldraw[black] (3,8) circle (2pt);
  \end{tikzpicture}
  \caption{Deformation of symplectic and complex structures.}
  \label{fig: Deformation}
\end{figure}

The initial manifold is $(X, \omega_{0}, J_{0})$. We fix a toric fiber $L(u)$ where $u$ is an interior point of the moment polytope of $X$. Then all $n_{\beta}$ are well-defined and have been computed in \cite{CL}. In particular there are finitely many nonzero $n_{\beta}$ and all such $\beta$ have the following form
$$
\beta=\beta_{j}+\sum_{i=1}^{k}s_{i}D_{i}
$$
where $\beta_{j}$ is a basic disk class and $D_{1}, \cdots, D_{n}$ is a chain of -2 toric divisors. See Definition 1.1 and Theorem 1.2 in \cite{CL} for more details of their computation. Then we choose a small neighborhood $U$ of the maximal chain of these -2 toric divisors such that $u\not\in U$ to get the above family of symplectic structures.

\begin{proposition}
If $\omega_{0}(D_{i})$ is small enough for all $1\leq i\leq n$. Then the mapping degree of
$$
ev: \mathcal{M}_{1} \left(\omega_{t}; L(u); \beta \right) \rightarrow L(u)
$$
is well-defined for all $t$ and independent of $t$, with respect to a certain family of almost complex structures $J_{t}$.
\end{proposition}
\begin{proof}
The idea is to use a cobordism argument. It suffices to show that for any time $t$ there a breaking of disks or spheres will not appear.

Let $J_{t}$ be a smooth family of almost complex structures such that $J_{t}$ is $\omega_{t}$-tame. The existence of such a family $J_{t}$ can be obtained as follows. Let $\mathcal{J}_{\tau}\left(\omega_{t}\right)$ be the space of $\omega_{t}$-tame almost complex structures, then for $\abs{t'-t''}$ small enough we have that
$$
\bigcap_{t'\leq t\leq t''}\mathcal{J}_{\tau}\left(\omega_{t}\right)\neq \emptyset, \quad \forall t'<t''
$$
Let $0=t_{0}<t_{1}<\cdots <t_{n}=1$ be a division of $[0, 1]$ such that
$$
\bigcap_{t_{k}\leq t\leq t_{k+1}}\mathcal{J}_{\tau}\left(\omega_{t}\right)\neq \emptyset.
$$
Such a division exists because $\mathcal{J}_{\tau}\left(\omega_{t}\right)$ is open in the space of almost complex structures and our path is compact. We start with $\omega_{0}$ and $J_{0}$ which is the toric complex structure on $X$. For a fixed division as above we pick $J_{t_{1}}\in \cap_{0<t<t_{1}}\mathcal{J}\left(\omega_{t}\right)$ and connect $J_{0}$ with $J_{t_{1}}$ by a path $J_{s}$ in $\mathcal{J}_{\tau}\left(\omega_{0}\right)$. When we deform the symplectic structure from $\omega_{0}$ to $\omega_{t_{1}}$ the almost complex structure $J_{t_{1}}$ is unchanged. Hence we get a cobordism from $\mathcal{M}_{1} \left(\omega_{0}; J_{0}; \beta \right)$ to $\mathcal{M}_{1} \left( \omega_{0}; J_{t_{1}}; \beta \right)$. Next we get a cobordism from $\mathcal{M}_{1} \left(\omega_{0}; J_{t_{1}}; \beta \right)$ to $\mathcal{M}_{1} \left(\omega_{t_{1}}; J_{t_{1}}; \beta \right)$ since $J_{t_{1}}$ is $\omega_{t}$-tame for all $0\leq t\leq t_{1}$. That is, we first fix symplectic structure and deform the almost complex structure then fix a common compatible almost complex structure and deform the symplectic structure, see Figure 2. We extend this procedure step by step to get a cobordism from $\mathcal{M}_{1} \left( \omega_{0}; J_{0}; \beta \right)$ to $\mathcal{M}_{1} \left( \omega_{J}; J_{1}; \beta \right)$ where $J_{0}$ is the toric complex structure and $J_{1}$ is a generic almost complex structure compatible with $\omega_{1}$. During this process we can make that $J_{t}=J_{0}$ outside $U$ just by picking smaller open subset $V$ in $U$ and extend the almost complex structures.

The initial manifold is $(X, \omega_{0}, J_{0})$. Since $J_{0}$ is toric we can apply Cho-Oh's classification and regularity theorems in \cite{CO}. So there is no holomorphic disk of nonpositive Maslov index hence all our moduli spaces carry a fundamental cycle and the mapping degrees are well-defined. By dimension formulas
$$
dim \mathcal{M}(\beta, J)=2+\mu(\beta)-3=\mu(\beta)-1, \quad \beta\in H_{2}(X, L(u))
$$
and
$$
dim \mathcal{M}(D, J)=4+2c_{1}(D)-6=-2, \quad D\in H_{2}(X)
$$
we can choose our one-parameter family to be generic to avoid disk bubbles of strictly negative Maslov index and sphere bubbles.

Next we prove that during the deformation there is no disk bubble of Maslov index 0. Let
$$
\beta=\beta_{j}+\sum_{i=1}^{k}s_{i}D_{i}
$$
then we have that
\begin{equation}
\begin{split}
\omega_{t}(\beta)=& \omega_{t}(\beta_{j})+\sum_{i=1}^{n}s_{i}\omega_{t}(D_{i})\\
\leq & l_{j}(u)+\sum_{i=1}^{n}s_{i}\omega_{0}(D_{i})
\end{split}
\end{equation}
because in our deformation the symplectic energy of class $D_{i}$ decreases to zero. Here $l_{j}$ is the $j$th affine function defining the moment polytope where $D_{j}$ corresponds to $l_{j}=0$. Now suppose that there is a disk bubble of Maslov index 0. That is, our holomorphic disk of class $\beta$ splits into two disks of classes $\beta'$ and $\beta''$ respectively where
$$
\beta=\beta' +\beta'', \quad \mu(\beta')=2, \quad \mu(\beta'')=0.
$$
The image of those singular disks contains several disk and sphere components. There exist at least one disk component which intersect the chain of $D_{i}$ for homological reason. This main component consumes most of the symplectic energy, which is roughly $l_{j}(u)$. Other components have very small symplectic energy, which is less than $\sum_{i=1}^{n}s_{i}\omega_{0}(D_{i})$.

We only have finitely many tuples $\lbrace s_{i}\rbrace$ as a priori such that $n_{\beta}\neq 0$. Therefore we can make $\sum_{i=1}^{n}s_{i}\omega_{0}(D_{i})$ uniformly small such that the image of other components can not escape the undeformed region. Note that we can shrink the neighborhood $U$ very small such that all holomorphic disks with boundary on $L(u)$ need some energy to reach $U$, which is larger than $\sum_{i=1}^{n}s_{i}\omega_{0}(D_{i})$. Moreover we can assume that the image of other components lies in a neighborhood of $L(u)$ which does not intersect any toric divisor. In this neighborhood we have toric complex structure and Cho-Oh's Maslov index formula in \cite{CO} shows that the sum of other components has Maslov index 0. But on the other hand there is no holomorphic disk of Maslov index 0 with respect to the toric complex structure. Therefore the main component is the whole disk class $\beta$. It follows that when we deform the symplectic structures there is no disk bubble and the moduli space does not split. Similarly the complex structure is only deformed in a small neighborhood such that the change of energy can be also uniformly controlled. In conclusion, firstly we know that for each $t$ the mapping degree is well-defined and secondly those mapping degrees do not depend on $t$.
\end{proof}

We just use the hyperK\"ahler ALE metrics to get a family of symplectic forms which are Euclidean outside a large compact set then pick the almost complex structures. The pair $(\omega_{t}, J_{t})$ is not necessarily to be hyperK\"ahler.

This deformation invariance of $n_{\beta}$ gives us necessary ingredients to compute the disk potential function in the smoothing $\hat{X}$. We remark that how the open Gromov-Witten invariants can be related between orbifolds, resolution and smoothing is also studied in \cite{CCLT1}, \cite{CCLT2} \cite{L}, \cite{LLW} from the perspective of SYZ mirror symmetry and the perspective of the open crepant resolution conjecture. In \cite{CKO} the conifold transition is studied in some monotone symplectic manifold and families of nondisplaceable Lagrangian tori have also been found.

\subsection{An example: $X_{2}$}
In this subsection we study the concrete example of $X_{2}$. Other cases of semi-Fano toric surfaces are similar and we just list the result in Appendix 2.

Let $X_{2}$ be the semi-Fano toric surface corresponding to the following moment polytope.

$l_{1}: u_{1} \geq 0;$

$l_{2}: u_{2} \geq 0;$

$l_{3}: t_{1}+t_{2}+2t_{3}-u_{1}-u_{2} \geq 0;$

$l_{4}: t_{1}+t_{3}-u_{2} \geq 0;$

$l_{5}: t_{1}+u_{1}-u_{2} \geq 0.$

Here $t_{i}$ are positive K\"ahler parameters. We first fix some parameters then do degeneration. When $2t_{1}=t_{2}+2t_{3}$ the smoothing of the degenerate surface is monotone and we can get a monotone Lagrangian torus fiber. For example if we take $t_{1}=t_{2}=2, t_{3}=1$ then we get a toric semi-Fano surface $X_{2}$ of which the moment polytope is cut out by

$l_{1}: u_{1} \geq 0;$

$l_{2}: u_{2} \geq 0;$

$l_{3}: 6-u_{1}-u_{2} \geq 0;$

$l_{4}: 3-u_{2} \geq 0;$

$l_{5}: 2+u_{1}-u_{2} \geq 0.$

Self-intersection numbers are labelled near toric divisors in Figure 3. The -2-curve can be collapsed in a toric way by considering parallel translating the affine function $l_{4}: 4-\alpha-u_{2} \geq 0$ where the parameter $\alpha\in [0,1]$. The symplectic energy of the -2-curve is $2\pi\cdot 2\alpha$. For notational simplicity we just write the energy of the -2-curve to be $2\alpha$. This does not affect our final potential function in $\hat{X}_{2}$ since we are taking the limit $\alpha\rightarrow 0$. When $\alpha=0$ it corresponds to the orbifold moment polytope in Figure 3. For example if we want to compute the potential functions of all interior points where $u_{2}\leq 3$ then we choose $\alpha$ very small and deform a neighborhood of the -2-curve such that the symplectic and complex structure are unchanged inside the region where $u_{2}\leq 3.9$. Those numbers are not exactly computed or estimated but just to morally show that we can choose $\alpha$ very small to use Proposition 3.7.

Let $D_{i}$ be the toric divisor corresponding to $l_{i}=0$. The disk potential function was computed in \cite{CL} and we rewrite it in the notion of \cite{FOOO3} as follows. Let $u=(u_{1}, u_{2})$ be an interior point in the polytope and $L(u)$ be the corresponding torus fiber. Then we have that
\begin{equation}
\begin{split}
\mathfrak{PO}_{\mathfrak{b}}^{L(u)}(y_{1}, y_{2}) &=e^{a_{1}}y_{1}T^{u_{1}}+e^{a_{2}}y_{2}T^{u_{2}}+e^{a_{3}}y_{1}^{-1}y_{2}^{-1}T^{6-u_{1}-u_{2}}\\
&+(e^{a_{4}}+e^{a_{3}-a_{4}+a_{5}}T^{2\alpha})y_{2}^{-1}T^{4-\alpha-u_{2}}+e^{a_{5}}y_{1}y_{2}^{-1}T^{2+u_{1}-u_{2}}
\end{split}
\end{equation}
where $\mathfrak{b}=\sum_{i=1}^{5} a_{i}PD([D_{i}]) \in H^{2}(X_{2}; \Lambda_{0})$ is a bulk deformation.

\begin{figure}
  \begin{tikzpicture}
  \path [fill=lightgray, lightgray] (-6,0)--(-6,1)--(-5.5,1.5)--(-4.5,1.5)--(-3,0);
  \draw (-6,0)--(-6,1)--(-5.5,1.5)--(-4.5,1.5)--(-3,0)--(-6,0);
  \node at (-5,0.75) {$X_{2}$};
  \node [above] at (-5,1.5) {-2};
  \node [above left] at (-5.75,1.25) {-1};
  \node [left] at (-6,0.5) {-1};
  \node [below] at (-5,0) {1};
  \node [above right] at (-4,1) {0};
  \draw [->] (-2,1)--(-1,1);
  \path [fill=lightgray, lightgray] (0,0)--(0,1)--(1,2)--(3,0);
  \draw (0,0)--(0,1)--(1,2)--(3,0)--(0,0);
  \draw (1,1)--(1,1.5);
  \node at (1,0.5) {$X_{2,0}$};
  \draw [dashed] (0,1.5) --(2,1.5);
  \node [right] at (2,1.5) {$u_{2}=3$};
  \filldraw[black] (1,1) circle (1pt);
  \end{tikzpicture}
  \caption{Toric degeneration from $X_{2}$ to $X_{2,0}$.}
  \label{fig: Moment polytope 1}
\end{figure}
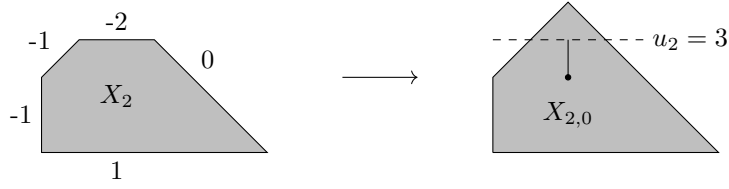

When $\alpha$ is small enough we deform the symplectic and complex structures in a neighborhood of the -2-curve to get $\hat{X}_{2}$. At the end of the deformation the -2-curve becomes a Lagrangian sphere which has symplectic energy zero. By Proposition 3.7 all the one-point open Gromov-Witten invariants are preserved. Hence for an interior point $u$ inside $u_{2}\leq 3$ the potential function of $L(u)$ is formally setting $\alpha=0$ in (3.2), which is
\begin{equation}
\begin{split}
\mathfrak{PO}_{\mathfrak{b}}^{L(u)}(y_{1}, y_{2}) &=e^{a_{1}}y_{1}T^{u_{1}}+e^{a_{2}}y_{2}T^{u_{2}}+e^{a_{3}}y_{1}^{-1}y_{2}^{-1}T^{6-u_{1}-u_{2}}\\
&+(e^{a_{4}}+e^{a_{3}-a_{4}+a_{5}})y_{2}^{-1}T^{4-u_{2}}+e^{a_{5}}y_{1}y_{2}^{-1}T^{2+u_{1}-u_{2}}.
\end{split}
\end{equation}

Next we consider the points on the segment $u_{1}=2, 2\leq u_{2}\leq 3$. When $u_{1}=u_{2}=2$ the energy of each term is 2 so we find that $L(u)=L(2, 2)$ is a monotone Lagrangian torus with minimal Maslov number 2. Let the bulk deformation be $\mathfrak{b}=aPD([D_{4}])$ then the potential function of $L(u)$ is
$$
T^{-2}\mathfrak{PO}_{\mathfrak{b}}^{L(u)}(y_{1}, y_{2})=y_{1}+y_{2}+y_{1}^{-1}y_{2}^{-1}+(e^{a}+e^{-a})y_{2}^{-1}+y_{1}y_{2}^{-1}.
$$
The critical points equations are
$$
\left\{
\begin{aligned}
1-y_{1}^{-2}y_{2}^{-1}+y_{2}^{-1} &=0\\
1-y_{1}^{-1}y_{2}^{-2}-(e^{a}+e^{-a})y_{2}^{-2}-y_{1}y_{2}^{-2} &=0
\end{aligned}
\right.\notag.
$$
By simplifying above equations we get
$$
\left\{
\begin{aligned}
y_{2}^{-1}(y_{1}^{-2}-1) &=1\\
y_{1}^{5}+(e^{a}+e^{-a}-1)y_{1}^{4}+y_{1}^{3}+2y_{1}^{2}-1 &=0
\end{aligned}
\right.\notag.
$$
Setting $e^{a}$ as a generic complex number we can solve $y_{1}$ and $y_{2}$. For example when
\begin{equation}
\quad e^{a}+e^{-a}=3
\end{equation}
there are 5 different solutions $(y_{1}, y_{2})\in \mathbb{C}^{2}$, which can be checked with the help of the software Mathematica. Therefore this monotone Lagrangian torus has nontrivial deformed Floer cohomology and it is nondisplaceable. Moreover the determinant of the Hessian matrix is
\begin{equation}
\begin{split}
detHess=& 3y_{1}^{-4}y_{2}^{-4}+(e^{a}+e^{-a})y_{1}^{-3}y_{2}^{-4}+6y_{1}^{-2}y_{2}^{-4}-y_{2}^{-4}\\
=& [3y_{1}^{-4}+4(e^{a}+e^{-a})y_{1}^{-3}+6y_{1}^{-2}-1]y_{2}^{-4}.
\end{split}
\end{equation}
We can check that it is nonzero at above 5 solutions $(y_{1}, y_{2})$ by just plugging in.

For the above monotone fiber $L(u)$, if we do not choose bulk deformation to perturb the potential function there are still critical points which shows that it is nondisplaceable. In that case the critical points equations become
$$
\left\{
\begin{aligned}
y_{2}^{-1}(y_{1}^{-2}-1) &=1\\
y_{1}^{5}+y_{1}^{4}+y_{1}^{3}+2y_{1}^{2}-1=(y_{1}+1)^{2}(y_{1}^{3}-y_{1}^{2}+2y_{1}-1) &=0
\end{aligned}
\right.\notag.
$$
The solution $y_{1}=-1$ of the second equation does not give a solution $y_{2}$ in the first equation. So we only have three different solutions and they are all nondegenerate. We perturb the potential function to get 5 nondegenerate critical points to make the deformed Kodaira-Spencer map an isomorphism, which will be used later to produce quasi-morphisms.

Next we consider other points on $u_{1}=2, 2\leq u_{2}\leq 3$ where the potential function of $L(u)$ is
$$
T^{u_{2}-4}\mathfrak{PO}_{\mathfrak{b}}^{L(u)}(y_{1}, y_{2})=y_{1}^{-1}y_{2}^{-1}+(e^{a}+e^{-a})y_{2}^{-1}+y_{1}y_{2}^{-1}+y_{1}T^{u_{2}-2}+y_{2}T^{2u_{2}-4}.
$$
The critical points equations are
$$
\left\{
\begin{aligned}
-y_{1}^{-2}y_{2}^{-1}+y_{2}^{-1}+T^{u_{2}-2} &=0\\
-y_{1}^{-1}y_{2}^{-2}-(e^{a}+e^{-a})y_{2}^{-2}-y_{1}y_{2}^{-2}+T^{2u_{2}-4} &=0
\end{aligned}
\right.\notag.
$$
By setting $y_{2}=1$ we simplify above equations to
$$
\left\{
\begin{aligned}
(1+T^{u_{2}-2})^{-\frac{1}{2}} &=y_{1}\\
e^{2a}+(y_{1}+y_{1}^{-1}-T^{2u_{2}-4})e^{a}+1 &=0
\end{aligned}
\right.\notag.
$$
Then we can solve $e^{a}\in \Lambda_{0}-\Lambda_{+}$ by the following lemma.

\begin{lemma}
Consider the polynomial equation
$$
X^{n}+a_{n-1}X^{n-1}+\cdots +a_{1}X+a_{0}=0
$$
where the coefficients $a_{n-1}, \cdots, a_{0}\in \Lambda_{0}$ and $0\neq a_{0}\in \Lambda_{0}- \Lambda_{+}$. Then it has at least one solution in $\Lambda_{0}- \Lambda_{+}$.
\end{lemma}
\begin{proof}
Let $\Lambda$ be the fraction field of $\Lambda_{0}$. It is algebraically closed when the ground field of our Novikov rings are algebraically closed. In this note we always assume that the ground field is $\mathbb{C}$. Therefore the above polynomial equation has at least one solution $X=f(T)/g(T)$ in $\Lambda$, where $f(T), g(T)\in \Lambda_{0}$. We divide common powers of the $T$ variable and assume that at least one of $f(T)$ and $g(T)$ is in $\Lambda_{0}- \Lambda_{+}$.

By plugging $X=f(T)/g(T)$ in the equation we have
$$
f(T)^{n}+ \sum_{k=1}^{n-1} a_{k}g(T)^{n-k}f(T)^{k} +a_{0}g(T)^{n}=0.
$$
If both $f(T), g(T) \in \Lambda_{0}- \Lambda_{+}$ then $X=f(T)/g(T)\in \Lambda_{0}- \Lambda_{+}$ since elements in $\Lambda_{0}- \Lambda_{+}$ are units. If $f(T) \in \Lambda_{0}- \Lambda_{+}$ but $g(T)$ not then we have a contradiction since $f(T)^{n}\in \Lambda_{0}- \Lambda_{+}$ but the other two terms are in $\Lambda_{+}$. If $g(T) \in \Lambda_{0}- \Lambda_{+}$ but $f(T)$ not then we have a contradiction since $a_{0}g(T)^{n}\in \Lambda_{0}- \Lambda_{+}$ but the other two terms are in $\Lambda_{+}$. Therefore both $f(T), g(T) \in \Lambda_{0}- \Lambda_{+}$ and their quotient is in $\Lambda_{0}- \Lambda_{+}$.
\end{proof}

Therefore we get solutions $(y_{1}, y_{2}, a)$ to the critical points equation which show that all $L(u)$ with $u_{1}=2, 2\leq u_{2}\leq 3$ are nondisplaceable.

\begin{remark}
The above computation shows that there exists at least one solution $(y_{1}, y_{2}=1, e^{a})$. We may have more ``formal'' critical points since our equations are polynomials. However we can not make sure that these ``formal'' critical points are in $\Lambda_{0}-\Lambda_{+}$, which we need to take logarithm to get weak bounding cochains.
\end{remark}

The computation in the case of non-monotone $\hat{X}_{2}$ is essentially the same with the monotone case. We just list the result here to show that our method does not require the ambient symplectic manifold to be monotone.

We fix $t_{1}=t_{2}=t_{3}=2$ and redo the computation. In this case the moment polytope in toric degeneration is cut out by

$l_{1}: u_{1} \geq 0;$

$l_{2}: u_{2} \geq 0;$

$l_{3}: 8-u_{1}-u_{2} \geq 0;$

$l_{4}: 5-\alpha-u_{2} \geq 0;$

$l_{5}: 2+u_{1}-u_{2} \geq 0.$

By deforming a neighborhood of the -2-curve such that all the one-point open Gromov-Witten invariants of points inside $u_{2}\leq 3$ are preserved we get $\hat{X}_{2}$. The potential function of $L(u)$ in $\hat{X}_{2}$ is
\begin{equation}
\begin{split}
\mathfrak{PO}_{\mathfrak{b}}^{L(u)}(y_{1}, y_{2}) &=e^{a_{1}}y_{1}T^{u_{1}}+e^{a_{2}}y_{2}T^{u_{2}}+e^{a_{3}}y_{1}^{-1}y_{2}^{-1}T^{8-u_{1}-u_{2}}\\
&+(e^{a_{4}}+e^{a_{3}-a_{4}+a_{5}})y_{2}^{-1}T^{5-u_{2}}+e^{a_{5}}y_{1}y_{2}^{-1}T^{2+u_{1}-u_{2}}.
\end{split}
\end{equation}
We consider the points on the segment $u_{1}=3, 2.5\leq u_{2}\leq 3$. By the same argument in previous subsection and the help of Lemma 3.8 we get solutions to the critical points equations. This shows that all $L(u)$ with $u_{1}=3, 2.5\leq u_{2}\leq 3$ are nondisplaceable in $\hat{X}_{2}$. However there is no monotone fiber in this family.

Next we compute the image of the Kodaira-Spencer map in the monotone case and find quasi-morphisms induced by the bulk-deformed Floer cohomology, mainly following Part 5 in \cite{FOOO5}.

The Kodaira-Spencer map is a ring homomorphism from the big quantum cohomology ring to the Jacobian ring of the potential function with bulk deformations. In the case of toric fibers \cite{FOOO4} it is proved always to be an isomorphism. Since the smoothing symplectic structure is not toric we need to choose a particular bulk deformation such that the Kodaira-Spencer map is an isomorphism and the potential function is Morse, which implies that both the Jacobian ring and big quantum cohomology ring are semi-simple. We set $y_{i}=y_{i}T^{u_{i}}$ in (3.3) and get the potential function
\begin{equation}
\begin{split}
\mathfrak{P}_{\mathfrak{b}}(y_{1}, y_{2}) &=e^{a_{1}}y_{1}+e^{a_{2}}y_{2}+e^{a_{3}}y_{1}^{-1}y_{2}^{-1}T^{6}\\
&+(e^{a_{4}}+e^{a_{3}-a_{4}+a_{5}})y_{2}^{-1}T^{4}+e^{a_{5}}y_{1}y_{2}^{-1}T^{2}.
\end{split}
\end{equation}
which is independent of $u$. Then by differentiating (3.7) we get
$$
\mathfrak{ks}_{\mathfrak{b}}(PD([D_{i}]))=e^{a_{i}}y_{i}, \quad i=1,2;
$$
$$
\mathfrak{ks}_{\mathfrak{b}}(PD([D_{3}]))=e^{a_{3}}y_{1}^{-1}y_{2}^{-1}T^{6}+e^{a_{3}-a_{4}+a_{5}}y_{2}^{-1}T^{4};
$$
$$
\mathfrak{ks}_{\mathfrak{b}}(PD([D_{4}]))=(e^{a_{4}}-e^{a_{3}-a_{4}+a_{5}})y_{2}^{-1}T^{4};
$$
$$
\mathfrak{ks}_{\mathfrak{b}}(PD([D_{5}]))=e^{a_{3}-a_{4}+a_{5}}y_{2}^{-1}T^{4}+e^{a_{5}}y_{1}y_{2}^{-1}T^{2}.
$$
This can be seen as a definition of the bulk-deformed Kodaira-Spencer maps, for more details we refer to Part 5 in \cite{FOOO5}.

By the computation in the monotone case we have 5 nondegenerate critical points of potential function of the monotone fiber with respect to some bulk deformation $\mathfrak{b}$. This induces 5 nondegenerate critical points of (3.7). Then by Proposition 20.23 in \cite{FOOO5} we know that the Jacobian ring factorizes as a direct product. Moreover we can directly check that $e^{a_{4}}-e^{a_{3}-a_{4}+a_{5}}\neq 0$ hence the image of the $\mathfrak{b}$-deformed Kodaira-Spencer map is 5-dimensional. Therefore the $\mathfrak{b}$-deformed Kodaira-Spencer map is an isomorphism and the $\mathfrak{b}$-deformed big quantum cohomology ring of $\hat{X}_{2}$ is semi-simple.
$$
\begin{tikzcd}
QH_{\mathfrak{b}}(\hat{X}_{2}; \Lambda) \arrow[d, "i"] \arrow[r, "\mathfrak{ks}_{\mathfrak{b}}"] & Jac(\mathfrak{PO}_{\mathfrak{b}}; \Lambda)\cong \Lambda^{5}\\
HF(L(u); \mathfrak{b}, b)
\end{tikzcd}
$$
By a similar proof of Theorem 23.4 in \cite{FOOO5} there is an idempotent $e\in QH_{\mathfrak{b}}(\hat{X}_{2}; \Lambda)$ such that $i(e)$ is nonzero in $HF(L(u); \mathfrak{b}, b)$. Then by Theorem 18.8 in \cite{FOOO5} we get a Calabi quasi-morphism on $\widetilde{Ham}(\hat{X}_{2}, \omega)$, which is the universal cover of the Hamiltonian diffeomorphism group.

Next we study how critical points vary with respect to the change of bulk deformations to construct a one-dimensional family of quasi-morphisms on $\widetilde{Ham}(\hat{X}_{2}, \omega)$. The proof is based on Lemma 24.13 in \cite{FOOO5} and we sketch it as follows. When there is no bulk deformation then (3.7) has exactly 3 nonzero critical points and they are nondegenerate by previous computation. When the bulk deformation $\mathfrak{b}(a)=a=(0, 0, 0, a_{4}, 0)$ is generic in $\mathbb{C}$ then (3.7) has 5 nondegenerate critical points. This tells us some of the critical points go to either infinity or zero when the bulk deformation goes to zero. Then consider the solution space
$$
S_{0}=\lbrace (a; y_{1}, y_{2})\in \mathbb{C}^{3}\mid (y_{1}, y_{2})\quad\text{is a critical point of}\quad \mathfrak{PO}_{\mathfrak{b}(a)} \rbrace
$$
and its Zariski closure $S$ in $\mathbb{C}\times \mathbb{C}P^{1} \times \mathbb{C}P^{1}$. We have a projection $\pi: S\rightarrow \mathbb{C}$ onto the first parameter. A generic fiber of $\pi$ has 5 points and the fiber $\pi^{-1}(0)$ intersects $S_{0}$ at 3 points. Since these 3 critical points are simple we have that $\pi^{-1}(0)-S_{0}\neq \emptyset$ and the points in $\pi^{-1}(0)-S_{0}$ are either infinity or zero. We write down the Laurent series $(a(w); y_{1}(w), y_{2}(w))$ near these points. Then $(y_{1}(T^{\delta}), y_{2}(T^{\delta}))$ are critical points of $\mathfrak{PO}_{\mathfrak{b}(a(T^{\delta}))}$. Moreover for a given $a(T^{\delta})$ there are 5 choices of $y_{1}(T^{\delta}), y_{2}(T^{\delta}))$. In summary for each small $\delta\in (0, \epsilon)$ we have a quasi-morphism induced from the nontrivial Floer cohomology
$$
HF(L(2, 2+\delta); \mathfrak{b}(\delta), b=(y_{1}(\delta), y_{2}(\delta))).
$$
Note that all $L(2, 2+\delta)$ are mutually disjoint so we have a one-dimensional family of linearly independent quasi-morphisms.

\subsection{Relation with other's work}
As we have seen there are certain K\"ahler parameters of $X$ which represent the symplectic energy of different divisors. Those parameters can be carefully chosen such that $\hat{X}$ is a monotone symplectic manifold. Great work of the classification of monotone symplectic 4-manifolds has been done, see \cite{LL}, \cite{M1}, \cite{M2}, \cite{OO1}, \cite{OO2}, \cite{TC1}, \cite{TC2}, \cite{TC3} for an incomplete list. Thereby our surface $\hat{X}$ is isomorphic to one of the k-points blowup of $\mathbb{C}P^{2}$ with a monotone symplectic structure, also known as the symplectic del Pezzo surfaces. In particular, by counting Betti numbers there are following isomorphisms
\begin{enumerate}
\item $\hat{X}_{2}\cong Bl_{2}\mathbb{C}P^{2}$;
\item $\hat{X}_{3}, \hat{X}_{4}, \hat{X}_{5}\cong Bl_{3}\mathbb{C}P^{2}$;
\item $\hat{X}_{6}, \hat{X}_{7}\cong Bl_{4}\mathbb{C}P^{2}$;
\item $\hat{X}_{8}, \hat{X}_{9}, \hat{X}_{10}\cong Bl_{5}\mathbb{C}P^{2}$.
\end{enumerate}
In these cases we have a unique fiber $L(u)$ in each degeneration that is a monotone Lagrangian torus, which is one of Vianna's tori in \cite{V2}.

Vianna kindly points out the following relations between our toric degeneration pictures and his almost toric base diagrams in \cite{V2}.
\begin{enumerate}
\item $X_{3,0}$ corresponds to $A_5$ in Figure 16;
\item $X_{4,0}$ corresponds to $A_4$ in Figure 16;
\item $X_{5,0}$ corresponds to $A_3$ in Figure 16;
\item $X_{6,0}$ corresponds to $A_2$ in Figure 17;
\item $X_{7,0}$ corresponds to a mutation of $A_1$ in Figure 17, mutating one of the nodes in the cut with 3 nodes;
\item $X_{8,0}$ corresponds to $B_2$ in Figure 18;
\item $X_{9,0}$ corresponds to $B_1$ or $A_3$ in Figure 18;
\item $X_{10,0}$ corresponds to a mutation of $B_1$ in Figure 18, mutating one of the nodes in the cut with 3 nodes.
\end{enumerate}
The nondisplaceable fibers here are analogues of the $\beta_{i}$-type fibers in \cite{STV} where they discussed all possible Lagrangian tori of $\mathbb{C}P^{2}$ arising as fibers of almost toric fibrations.

Moreover, with K\"ahler parameters carefully chosen we have different degenerations to the same monotone symplectic manifold. For example $\hat{X}_{3}, \hat{X}_{4}, \hat{X}_{5}$ can be regarded as three different degenerations to $Bl_{3}\mathbb{C}P^{2}$, each of them giving a monotone Lagrangian torus $L_{3}, L_{4}, L_{5}$. These tori are different under Hamiltonian diffeomorphisms since they bound $11, 8 ,7$ families of holomorphic disks of Maslov index 2 respectively. Although $L_{3}, L_{4}, L_{5}$ have different potential functions they should have the same quantum period by Tonkonog's quantum period theorem in \cite{T}. Since the potential functions are explicitly known, one can easily check it numerically by a computer. Theoretically, those monotone tori are related under mutations and the potential functions can be obtained by Pascaleff-Tonkonog's wall-crossing formula in \cite{PT}.

\section{Lagrangian tori near a chain of Lagrangian spheres}
In this section we generalize the result in the case of del Pezzo surfaces to the case of a closed symplectic 4-manifold which contains a linear chain of Lagrangian spheres.

\subsection{Local computation in $A_{n}$-type Milnor fibers}
The existence of a one-parameter family of monotone Lagrangian tori $L_{n, \lambda}$ with nontrivial Floer cohomology in $A_{n}$-type Milnor fibers is known. Here we provide the computation of their potential functions with bulk deformation by using toric degeneration and admissible tuples in \cite{CL}. Also see \cite{LM} to count Maslov index 2 holomorphic disks by using Lefschetz fibrations and in \cite{AAK}, \cite{L} from the perspective of mirror symmetry.

The $A_{n}$-type singularity $X_{0}$ is a toric singularity with a moment polytope
$$
P=\lbrace l_{0}\geq 0\rbrace \cap \lbrace l_{n+1}\geq 0\rbrace
$$
where $l_{0}=u_{1}$ and $l_{n+1}=(n+1)u_{2}-nu_{1}$. We choose a toric resolution $X_{\alpha}$ such that it has a moment polytope
$$
P=\bigcap^{n+1}_{k=0}P_{k}=\bigcap^{n+1}_{k=0}\lbrace l_{k}\geq 0\rbrace
$$
defined by $n+2$ affine functions
\begin{enumerate}
\item $l_{0}= u_{1}$;
\item $l_{k}= ku_{2}-(k-1)u_{1}- \alpha_{k}, \quad k=1, \cdots, n$;
\item $l_{n+1}= (n+1)u_{2}-nu_{1}$.
\end{enumerate}
Let $D_{k}=\mu^{-1}\left(\partial P_{k}\right)$ be the toric divisors then $D_{1}, \cdots, D_{n-1}$ are exceptional spheres with first Chern number zero. Hence $X_{\alpha}$ is an open semi-Fano toric surface. Let $(u_{1}, u_{2})\in P$ and $L(u)$ be the Lagrangian torus fiber. We can compute the potential function of $L(u)$ in $X_{\alpha}$ by counting admissible tuples.

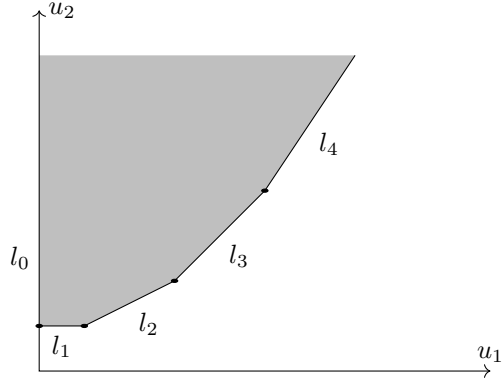
\begin{figure}
  \begin{tikzpicture}[xscale=0.6, yscale=0.3]
  \path [fill=lightgray, lightgray] (0,2)--(1,2)--(3,4)--(5,8)--(7,14)--(0,14)--(0,2);
  \draw [<->] (0,16) -- (0,0) -- (10,0);
  \draw (0,2)--(1,2)--(3,4)--(5,8)--(7,14);
  \node [right] at (0,16) {$u_{2}$};
  \node [above] at (10,0) {$u_{1}$};
  \node [left] at (0, 5) {$l_{0}$};
  \node [below] at (0.5, 2) {$l_{1}$};
  \node [below right] at (2, 3) {$l_{2}$};
  \node [below right] at (4, 6) {$l_{3}$};
  \node [below right] at (6, 11) {$l_{4}$};
  \filldraw[black] (0,2) circle (2pt);
  \filldraw[black] (1,2) circle (2pt);
  \filldraw[black] (3,4) circle (2pt);
  \filldraw[black] (5,8) circle (2pt);
  \end{tikzpicture}
  \caption{Moment polytope for $X_{\alpha}$ when $n=3$.}
  \label{fig: Toric resolution}
\end{figure}

\begin{definition}
We call a tuple $\lbrace k_{1}, \cdots, k_{m}\rbrace $ admissible with center $i$ if each $k_{l}$ is a nonnegative integer and
\begin{enumerate}
\item $k_{l}\leq k_{l+1}\leq k_{l}+1$ when $l<i$;
\item $k_{l}\geq k_{l+1}\geq k_{l}-1$ when $i\leq l$;
\item $k_{1}\leq 1$ and $k_{m}\leq 1$.
\end{enumerate}
\end{definition}

Then we have the following Theorem 1.2 from \cite{CL}.

\begin{theorem}
Let $X$ be a compact semi-Fano toric K$\ddot{a}$hler surface and $L$ be a Lagrangian fiber. Let $\beta \in \pi_{2}\left(X, L\right)$ be a class such that $\beta=\beta_{i} +\sum_{l}k_{l}D_{l}$, where $D_{l}$ is a chain of $-2$-curves in $X$ which are toric prime divisors. Then the one-point open Gromov-Witten invariant $n_{\beta}$ is either 1 or 0 according to whether the tuple $\lbrace k_{1}, \cdots, k_{m}\rbrace $ is admissible with center $i$ or not.
\end{theorem}

The above theorem works in the setting where $X$ is compact and it can be extended to our case where $X$ is open with a convex boundary. This is because the proof identifies the open Gromov-Witten invariants with some closed Gromov-Witten invariants of the chain of $-2$-curves, which has a local behavior. We refer to the proof of Theorem 1.2 in \cite{CL} for more details. Then by counting the admissible tuples we get the potential function in the toric resolution $X_{\alpha}$.

\begin{proposition}
Let $L(u)=L(u_{1}, u_{2})$ be a Lagrangian fiber of $X_{\alpha}$. Then the full potential function $\mathfrak{PO}^{L(u)}_{\alpha}\left(b\right)=\mathfrak{PO}_{\alpha}\left(u_{1}, u_{2}; y_{1}, y_{2} \right)$ is
\begin{equation}
\begin{split}
\mathfrak{PO}_{\alpha}\left( y_{1}, y_{2} \right)= &y_{1}T^{u_{1}}+y_{1}^{-n}y_{2}^{n+1}T^{-nu_{1}+(n+1)u_{2}}\\
&+\sum^{n}_{k=1}y_{1}^{1-k}y_{2}^{k}T^{(1-k)u_{1}+ku_{2}-\alpha_{k}}\left(1+H_{k}\left(T^{\alpha_{1}}, T^{\alpha_{2}}, \cdots, T^{\alpha_{n}}\right)\right).
\end{split}
\end{equation}
Here we set $y_{1}=e^{x_{1}}, y_{2}=e^{x_{2}}$ for $b=x_{1}e_{1}+x_{2}e_{2}\in H^{1}\left(L\left(u\right); \Lambda_{0}\right)$ and
$$
H_{k}\left(T^{\alpha_{1}}, T^{\alpha_{2}}, \cdots, T^{\alpha_{n}}\right)
$$
are polynomials in terms of $T^{\alpha_{1}}, T^{\alpha_{2}}, \cdots, T^{\alpha_{n}}$ corresponding to admissible tuples with center $k$.
\end{proposition}
\begin{proof}
We calculate the potential function by counting Maslov index 2 holomorphic disks. Let $\beta_{k}$ be meridian classes, that is, $\beta_{k}\cap D_{j}=\delta_{kj}$. Since $X_{\alpha}$ is semi-Fano, Theorem 11.1 in \cite{FOOO1} shows that a general class $\beta$ of Maslov index 2 with $\mathcal{M}_{1}\left(\beta\right)\neq\emptyset$ is of the following type
$$
\beta=\beta_{k}, \quad k=0, n+1;
$$
or
$$
\beta=\beta_{k}+s_{1}D_{1}+s_{2}D_{2}+\cdots +s_{n}D_{n}, \quad k=1, \cdots, n; \quad s_{k}\in \mathbb{Z}_{\geq 0}.
$$

The potential function has the formula
$$
\mathfrak{PO}^{L}(b)=\sum_{\mu_{L}(\beta)=2} n_{\beta}T^{\omega (\beta)} y_{1}^{e_{1}\left(\partial\beta\right)}y_{2}^{e_{2}\left(\partial\beta\right)}
$$
where
$$
\omega (\beta_{i})=l_{i}\left(u\right); \quad i=0, \cdots, n+1;
$$
$$
\omega (D_{i})=\alpha_{i}; \quad i=1, \cdots, n;
$$
$$
e_{j}\left(\beta_{i}\right)=\partial l_{i}\big/\partial u_{j}; \quad i=0, \cdots, n+1; \quad j=1, 2.
$$
and $n_{\beta}$ is the one-point open Gromov-Witten invariant. The first two terms in (4.1) come from the classes $\beta_{0}$ and $\beta_{n+1}$. Other terms come from classes
$$
\beta=\beta_{k}+s_{1}D_{1}+s_{2}D_{2}+\cdots +s_{n}D_{n}
$$
where we apply Theorem 4.2 on admissible tuples. The number $n_{\beta}$ is either one or zero according to whether the tuple $\lbrace s_{1}, \cdots, s_{n}\rbrace$ is admissible or not. When all $s_{k}=0$ we have the term $1$ in $1+H_{k}\left(T^{\alpha_{1}}, T^{\alpha_{2}}, \cdots, T^{\alpha_{n}}\right)$. Otherwise we have a term
$$
T^{s_{1}\alpha_{1}+\cdots +s_{n}\alpha_{n}}
$$
in $1+H_{k}\left(T^{\alpha_{1}}, T^{\alpha_{2}}, \cdots, T^{\alpha_{n}}\right)$. For a fixed integer $n$ we always have a finite collection of admissible tuples hence $H_{k}\left(T^{\alpha_{1}}, T^{\alpha_{2}}, \cdots, T^{\alpha_{n}}\right)$ is a polynomial. This finishes our calculation in the minimal resolution setting.
\end{proof}

Given positive integers $n$ and $k$ one can easily list all the admissible tuples with center $k$ just by definition. And there are finitely many admissible tuples when $n$ is fixed. This enables us to use the same deformation technique in subsection 3.3 to compute the potential function of $L(u)$ in the smoothing. Here we obtain the smoothing $X_{0}\left(\epsilon\right)$ of $X_{0}$ by replacing a neighborhood of the origin by a proper Milnor fiber, similar to what we did in Section 3, see Figure 5.

\begin{proposition}
Let $L(u)=L(u_{1}, u_{2})$ be a Lagrangian fiber of $X_{0}\left(\epsilon\right)$ under the coordinate identification between $X_{0}\left(\epsilon\right)$ and $X_{\alpha}$. Then the potential function $\mathfrak{PO}^{L(u)}\left(b\right)$ is
\begin{equation}
\begin{split}
\mathfrak{PO}\left( u_{1}, u_{2}; y_{1}, y_{2} \right)= &y_{1}T^{u_{1}}+y_{1}^{-n}y_{2}^{n+1}T^{-nu_{1}+(n+1)u_{2}}\\
&+\sum^{n}_{k=1}y_{1}^{1-k}y_{2}^{k}T^{(1-k)u_{1}+ku_{2}}\left(1+H_{k}\right).
\end{split}
\end{equation}
Here we set $y_{1}=e^{x_{1}}, y_{2}=e^{x_{2}}$ for $b=x_{1}e_{1}+x_{2}e_{2}\in H^{1}\left(L\left(u\right); \Lambda_{0}\right)$ and
$$
H_{k}=\lim_{\alpha\rightarrow 0} H_{k}\left(T^{\alpha_{1}}, T^{\alpha_{2}}, \cdots, T^{\alpha_{n}}\right)
$$
are limits of the polynomials as the energy parameters $\alpha=(\alpha_{1}, \cdots, \alpha_{n})$ tend to zero.
\end{proposition}
\begin{proof}
The proof is parallel to the proof of the deformation invariance in subsection 3.3. We connect two symplectic structures between the resolution and the smoothing by a path. Then we produce a cobordism between corresponding moduli spaces of holomorphic disks. Using the fact that there are finitely many admissible tuples we can give an energy estimate argument to exclude disk bubbles of Maslov index 0. Then this cobordism gives us the invariance of mapping degrees when all energy parameters in the resolution are sufficiently small. Therefore the potential function in the smoothing is obtained as a limit of the potential function in the resolution.
\end{proof}

\begin{figure}
  \begin{tikzpicture}[xscale=0.8, yscale=1]
  \draw (-2,6) to [out=45,in=135] (2,6);
  \draw (-2,6)--(-7,1);
  \draw (2,6)--(7,1);
  \draw (-4,4) to [out=300,in=240] (4,4);
  \draw (-3,5) to [out=300,in=240] (3,5);
  \draw (-6,2) to [out=300,in=240] (6,2);
  \draw [very thick] (0,0.6)--(0,1.8);
  \draw (-1,5.5)--(1,5.5);
  \draw (-1.6,5)--(-0.8,5.6);
  \draw (0.8,5.6)--(1.6,5);
  \node [below] at (0,5) {vanishing cycles};
  \node [right] at (2,6) {deforming region};
  \node [right] at (2.5,5.5) {$X_{\alpha}\rightarrow M_{n, \epsilon, R}$};
  \node [below] at (0,3.2) {boundary deforming region};
  \node [below] at (3,1.8) {standard region};
  \node [below] at (0,0.6) {Lagrangian tori $L(u)$};
  \node [below] at (0,0.2) {parameterized by an interval};
  \end{tikzpicture}
  \caption{Smoothing of the $A_{n}$-type singularity.}
  \label{fig: Smoothing}
\end{figure}
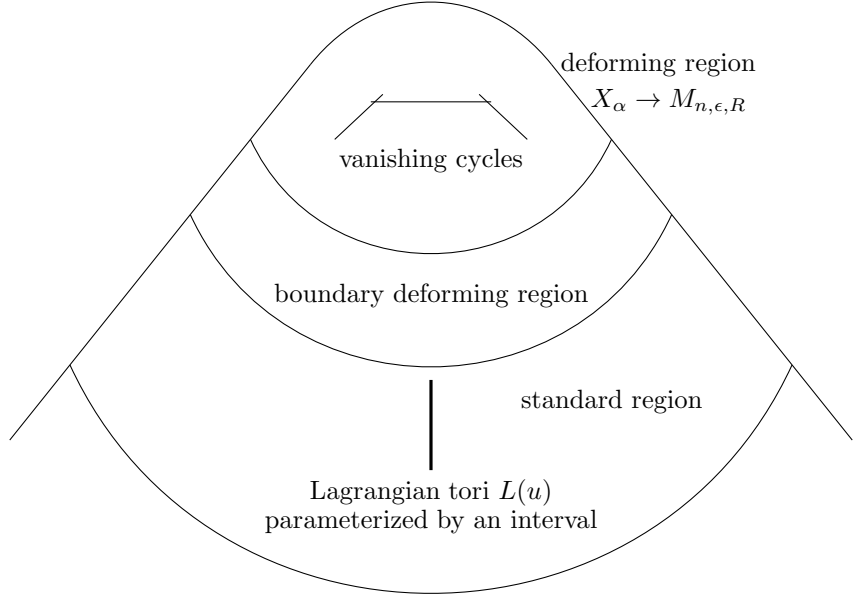

The coefficient $1+H_{k}$ is the number of admissible tuples of length $n$ with center $k$ hence it depends on both $n$ and $k$. Actually they are the binomial coefficients
$$
C(n, k)=\dfrac{n!}{(n-k)!k!}.
$$

\begin{lemma}
The number $1+H_{n, k}$ of admissible tuples of length $n$ with center $k$ is $C(n+1, k)$.
\end{lemma}
\begin{proof}
The first few numbers can be listed easily. For example in $A_{3}$-case we have
$$
1+H_{3, k}: \quad 1, 4, 6, 4, 1.
$$
And in in $A_{4}$-case we have
$$
1+H_{4, k}: \quad 1, 5, 10, 10, 5, 1.
$$
Then the rest of the proof follows the induction on $n$ with the binomial recursive formula $C(n+1, k)= C(n, k-1) + C(n, k)$.
\end{proof}

Using these integers we have the explicit expression of our potential function
\begin{equation}
\begin{split}
&\quad \mathfrak{PO}^{L(u)}\left(b\right)=\mathfrak{PO}\left(u_{1}, u_{2}; y_{1}, y_{2}\right)\\
&=\sum^{n+1}_{k=0}C(n+1, k) y_{1}^{1-k}y_{2}^{k}T^{(1-k)u_{1}+ku_{2}}.
\end{split}
\end{equation}

Next we look for its critical points when $u_{1}=u_{2}$. By setting partial derivatives to be zero we have
$$
0=\dfrac{\partial\mathfrak{PO}^{L(u)}}{\partial y_{1}}\left(y_{1}, y_{2}\right)=\sum^{n+1}_{k=0}C(n+1, k)(1-k) y_{1}^{-k}y_{2}^{k}T^{u_{1}}
$$
and
$$
0=\dfrac{\partial\mathfrak{PO}^{L(u)}}{\partial y_{2}}\left(y_{1}, y_{2}\right)=\sum^{n+1}_{k=0}C(n+1, k)k y_{1}^{1-k}y_{2}^{k-1}T^{u_{1}}.
$$
We claim that $y_{1}^{-1}y_{2}=-1$ is a solution to these two equations. For the second equation it follows from the fact that
$$
\left(\dfrac{d}{dx}(1+x)^{n+1}\right)\vert_{x=-1}=0.
$$
And the first one follows from the fact that
$$
(1+x)^{n+1}\vert_{x=-1}=0
$$
combined with the second one. Therefore when $u_{1}=u_{2}$ the potential functions always have critical points and the corresponding Lagrangian tori have nontrivial Floer cohomology. Note that the meridian classes $\beta_{k}$ generate $H_{2}\left(X_{0}\left(\epsilon\right), L(u)\right)$ and
$$
\mu_{L(u)}\left(\beta_{k}\right)=2, \quad \omega\left(\beta_{k}\right)=u_{1}=u_{2}, \quad \forall k=0, \cdots, n+1
$$
therefore when $u_{1}=u_{2}$ these tori are also monotone. The smoothing $X_{0}\left(\epsilon\right)$ is symplectomorphic to the Milnor fiber. In the following we do not distinguish these two and regard that the family of tori live inside the Milnor fiber.

As we can see from the solutions $y_{1}^{-1}y_{2}=-1$, the critical points here are degenerate and even non-isolated. In the global case we will use bulk deformation to perturb the potential function to deal with possible higher energy terms.

\subsection{Global computation in a symplectic 4-manifold}
Let $X$ be a symplectic 4-manifold which contains a linear chain of Lagrangian spheres. Then there exists a neighborhood which is symplectomorphic to the Milnor fiber. A family of our Lagrangian tori parameterized by a small interval naturally sits in this neighborhood. We will show that our Lagrangian tori are still nondisplaceable in $X$. First we consider the simple case that there is one Lagrangian sphere.

\begin{theorem}
(Corollary 1.2) Let $X$ be a closed symplectic 4-manifold which contains a Lagrangian 2-sphere. Then there is a one-parameter family of nondisplaceable Lagrangian tori in $X$ near this Lagrangian sphere.
\end{theorem}
\begin{proof}
Pick a neighborhood $U$ of the Lagrangian sphere which is symplectomorphic to the disk bundle of the cotangent bundle of a 2-sphere. Then we have a one-parameter family $L_{\lambda}=L_{1, \lambda}$ inside $U$ with $\lambda$ parameterized by an open interval. We study the potential function with bulk deformation of the torus $L_{\lambda}$ to show it is nondisplaceable in $X$.

The potential function of $L_{\lambda}$ counts Maslov index 2 holomorphic disks. We pick an almost complex structure which agrees the complex structure we use to compute the local potential function on $U$ and extend it to a generic one on $X$ such that the potential function is well-defined. Then the holomorphic disks of lowest energy are explicitly known due to the local computation.

Let $\mathfrak{b}=v PD([S_{s}])+ w PD([S_{n}])$ be a bulk deformation. Here $S_{s}$ is the fiber at south pole and $S_{n}$ is the fiber at north pole. The boundary of a fiber of the disk cotangent bundle is a torsion element in the boundary of the disk cotangent bundle. So both $S_{s}$ and $S_{n}$ represent rational cycles in $X$, which can be used as bulk deformations. We remark that when their boundaries are contractible in $X-U$ then we can regard them as integral cycles in $X$. Therefore by (2.1) the whole potential function is
$$
(e^{v}y_{1}+ (1+e^{v+w})y_{2}+ e^{w}y_{1}^{-1}y_{2}^{2})T^{\lambda} +P(y_{1}, y_{2}; v, w)T^{\mu}, \quad \lambda<\mu.
$$
where $P(y_{1}, y_{2}; v, w)$ come from the contribution of high energy disks. Setting $y_{1}=y_{2}=1$ then the critical point equations are
$$
\left\{
\begin{aligned}
e^{v}-e^{w}+ P_{1}(v, w)T^{\mu-\lambda} &=0\\
1+e^{v+w}+2e^{w}- P_{2}(v, w)T^{\mu-\lambda} &=0
\end{aligned}
\right.\notag
$$
The leading order term equations are
$$
\left\{
\begin{aligned}
e^{v}-e^{w} &=0\\
1+e^{v+w}+2e^{w} &=0
\end{aligned}
\right.\notag
$$
which has an isolated solution at $(-1, -1)$ of multiplicity 2.

To find solutions of the whole critical point equations we discuss in two cases.
\begin{enumerate}
\item If the symplectic form on $X$ is rational then all the energy parameters lie in a finitely generated semigroup in $\mathbb{R}$ hence Lemma 5.1 tells us there is always a solution in $\Lambda_{0}^{2}$ near $(-1, -1)$ to above equations, in the sense of non-Archimedean topology. Therefore we have a solution in $(\Lambda_{0}-\Lambda_{+})^{2}$ which shows that the deformed Floer cohomology is nontrivial.
\item In general the above equations have solutions if the tails $P_{1}, P_{2}$ have finite length, also by Lemma 5.1. Therefore we can truncate those equations to get solutions modulo large power. This shows that the deformed Floer cohomology is nonzero modulo any large energy hence our tori are nondisplaceable. But whether the full Floer cohomology is nontrivial or not is not known.
\end{enumerate}
By the above discussion we complete the proof.
\end{proof}

Next we study the general $A_{n}$-chain case, which is of the same spirit but combinatorially harder.

\begin{theorem}
Let $X$ be a closed symplectic 4-manifold which contains a linear chain of Lagrangian 2-spheres. Consider the Lagrangian embedding
$$
L_{n, \lambda}\hookrightarrow M_{n, \epsilon}\subset X
$$
then $L_{n, \lambda}$ is nondisplaceable in $X$ for all $\lambda$ in a possibly smaller open interval.
\end{theorem}
\begin{proof}
Let $D_{1}, \cdots, D_{n}$ be a linear chain of Lagrangian spheres with $n\geq 2$. We consider the bulk deformation by
$$
\mathfrak{b}=v PD(D_{1})+w PD(D_{n})\in H^{2}\left(X; \Lambda_{0}\right).
$$
Then by (2.1) and the local computation the leading terms of the full potential function with bulk deformation is
\begin{equation}
\mathfrak{PO}_{\mathfrak{b}, 0}\left(y_{1}, y_{2}\right)= y_{1}+\sum_{k=1}^{n}P_{k}\left(e^{v}, e^{w}\right)y_{1}^{1-k}y_{2}^{k}+ y_{1}^{-n}y_{2}^{n+1}
\end{equation}
where $P_{k}(e^{v}, e^{w})$ are Laurent polynomials. Since all terms have the same energy we omit the energy parameter here. The full potential function will be the leading terms plus higher energy terms. We follow the same idea in Theorem 4.6 to show the existence of critical points of the full potential function. Setting the partial derivatives equal to zero we get
\begin{equation}
0= \dfrac{\partial\mathfrak{PO}_{\mathfrak{b},0}}{\partial y_{1}}\left(y_{1}, y_{2}\right)= 1+\sum^{n}_{k=1}P_{k}(e^{v}, e^{w})(1-k)y_{1}^{-k}y_{2}^{k}-ny_{1}^{-n-1}y_{2}^{n+1}
\end{equation}
and
\begin{equation}
0=\dfrac{\partial\mathfrak{PO}_{\mathfrak{b},0}}{\partial y_{2}}\left(y_{1}, y_{2}\right)= \sum^{n}_{k=1}P_{k}(e^{v}, e^{w})ky_{1}^{1-k}y_{2}^{k-1}+(n+1)y_{1}^{-n}y_{2}^{n}.
\end{equation}

By the following Proposition 4.8 we know that for generic complex numbers $y_{1}, y_{2}$ the above critical point equation for the leading terms of the full potential function has nonzero isolated solutions $e^{v}, e^{w}$. Then by the same argument in Theorem 4.6 we complete the proof.
\end{proof}

We analyze the solution of the critical point equation for the leading terms of the full potential function. To find the critical points we give an algebraic argument without knowing all $P_{k}(e^{v}, e^{w})$ explicitly.

\begin{proposition}
With the same notations in Theorem 4.7, for generic complex numbers $y_{1}, y_{2}$ the critical point equation for the leading terms of the full potential function has nonzero isolated solutions $e^{v}, e^{w}$.
\end{proposition}
\begin{proof}
We assume that $n\geq 4$. When $n=2, 3$ one can check the computation directly as the case in Theorem 4.6, where $n=1$.

First we discuss some intersection number formulas to analyze $P_{k}(e^{v}, e^{w})$. Let $\beta_{1},\cdots, \beta_{n}$ be meridian classes corresponding the $A_{n}$-chain $D_{1}, \cdots, D_{n}$. For a disk class $\beta=\beta_{k}, k\neq 1, n$ we have
$$
D_{1}\cap \beta= D_{n}\cap \beta =0.
$$
Then consider a disk class
$$
\beta=\beta_{k}+s_{1}D_{1}+s_{2}D_{2}+\cdots +s_{n}D_{n}
$$
with nontrivial one-point open Gromov-Witten invariant. When $k=1$ we have that $s_{n}\leq \cdots \leq s_{2}\leq s_{1}\leq 1$ by the definition of admissible tuples. Hence we have
$$
D_{1}\cap \beta=\left\{
\begin{aligned}
 1, &\quad s_{1}=0\\
 -1, &\quad s_{1}=1, \quad s_{2}=0\\
 0, &\quad s_{1}=1, \quad s_{2}=1
\end{aligned}
\right.\notag
$$
and
$$
D_{n}\cap \beta=\left\{
\begin{aligned}
 0, &\quad s_{n-1}=0\\
 1, &\quad s_{n-1}=1, \quad s_{n}=0\\
 -1, &\quad s_{n-1}=1, \quad s_{n}=1
\end{aligned}
\right.\notag
$$
And symmetrically when $k=n$ we have
$$
D_{1}\cap \beta=\left\{
\begin{aligned}
 0, &\quad s_{2}=0\\
 1, &\quad s_{2}=1, \quad s_{1}=0\\
 -1, &\quad s_{2}=1, \quad s_{2}=1
\end{aligned}
\right.\notag
$$
and
$$
D_{n}\cap \beta=\left\{
\begin{aligned}
 1, &\quad s_{n}=0\\
 -1, &\quad s_{n}=1, \quad s_{n-1}=0\\
 0, &\quad s_{n}=1, \quad s_{n-1}=1
\end{aligned}
\right.\notag
$$
Lastly when $k\neq 1, n$ we have
$$
D_{1}\cap \beta=\left\{
\begin{aligned}
 0, &\quad s_{2}=0\\
 1, &\quad s_{2}=1, \quad s_{1}=0\\
 -1, &\quad s_{2}=1, \quad s_{1}=1\\
 0, &\quad s_{2}=2, \quad s_{1}=1
\end{aligned}
\right.\notag
$$
and
$$
D_{n}\cap \beta=\left\{
\begin{aligned}
 0, &\quad s_{n-1}=0\\
 1, &\quad s_{n-1}=1, \quad s_{n}=0\\
 -1, &\quad s_{n-1}=1, \quad s_{n}=1\\
 0, &\quad s_{n-1}=2, \quad s_{n}=1
\end{aligned}
\right.\notag
$$
Therefore $e^{v}, e^{w}$ only has degree one and negative one terms in $P_{k}(e^{v}, e^{w})$. That is, there are 9 possible terms
$$
e^{v}, e^{w}, e^{-v}, e^{-w}, e^{v+w}, e^{v-w}, e^{-v+w}, e^{-v-w}, constants
$$
and they all have nonnegative integer coefficients. We multiply $-e^{v+w}$ to (4.5) and multiply $e^{v+w}$ to (4.6) to get two polynomial equations
\begin{equation}
0= -e^{v+w}+e^{v+w} \sum^{n}_{k=1}P_{k}(e^{v}, e^{w})(k-1)y_{1}^{-k}y_{2}^{k}+ny_{1}^{-n-1}y_{2}^{n+1}e^{v+w}
\end{equation}
and
\begin{equation}
0=e^{v+w} \sum^{n}_{k=1}P_{k}(e^{v}, e^{w})ky_{1}^{1-k}y_{2}^{k-1}+(n+1)y_{1}^{-n}y_{2}^{n}e^{v+w}.
\end{equation}
We set $e^{v}=X$ and $e^{w}=Y$ and get
\begin{equation}
\left\{
\begin{aligned}
X^{2}(AY^{2}+BY+C)+X(DY^{2}+EY+F)+GY^{2}+HY+I &= 0\\
X^{2}(A'Y^{2}+B'Y+C')+X(D'Y^{2}+E'Y+F')+G'Y^{2}+H'Y+I' &= 0
\end{aligned}
\right.
\end{equation}
where $A, B, C, D, E, F, G, H, I, A', B', C, D, F', G', H', I'$ are nonconstant elements in $\mathbb{Z}[y_{1}, y_{2}, y_{1}^{-1}, y_{2}^{-1}]$.

By root formula we solve for $X$ from the first equation in term of $Y$ then plug in the second equation. We choose one of the root
$$
X=\dfrac{-(DY^{2}+EY+F)-\Delta}{2(AY^{2}+BY+C)}
$$
and get
\begin{equation}
\begin{split}
& [(DY^{2}+EY+F)+\Delta]^{2}(A'Y^{2}+B'Y+C')\\
+ & 2[-(DY^{2}+EY+F)-\Delta](D'Y^{2}+E'Y+F')(AY^{2}+BY+C)\\
+ & 4(G'Y^{2}+H'Y+I')(AY^{2}+BY+C)^{2}=0
\end{split}
\end{equation}
where
$$
\Delta =\sqrt{(DY^{2}+EY+F)^{2}-4(AY^{2}+BY+C)(GY^{2}+HY+I)}.
$$
After rearranging (4.10) we want to get a polynomial in $Y$ with nonzero solutions. So we need to check that the coefficients of the highest order term and the constant term are nonzero. By rearranging (4.10) we have that
\begin{equation}
\begin{split}
& [(DY^{2}+EY+F)^{2}+\Delta^{2}](A'Y^{2}+B'Y+C')\\
-& 2(DY^{2}+EY+F)(D'Y^{2}+E'Y+F')(AY^{2}+BY+C)\\
+& 4(G'Y^{2}+H'Y+I')(AY^{2}+BY+C)^{2}=\\
2& \Delta[(D'Y^{2}+E'Y+F')(AY^{2}+BY+C)-(DY^{2}+EY+F)(A'Y^{2}+B'Y+C')].
\end{split}
\end{equation}
Then we square both side of (4.11) to get a polynomial in $Y$. By direct calculation we know that the coefficient of the highest order term is
$$
16[A^{2}(DA'-D'A)(DG'-D'G)+A^{2}(G'A-A'G)^{2}]
$$
which depends on $A, A', D, D', G, G'$. Note that $A, A'$ are the coefficients of $e^{2v+2w}$ in (4.7) and (4.8), $D, D'$ are the coefficients of $e^{v+2w}$ and $G, G'$ the coefficients of $e^{2w}$. Next we calculate them out explicitly.

By above intersection number formulas we have that when $k=1$
$$
P_{1}(e^{v}, e^{w})=e^{v}+e^{-v}+e^{w}+e^{-w}+n-3
$$
and symmetrically when $k=n$
$$
P_{n}(e^{v}, e^{w})=e^{v}+e^{-v}+e^{w}+e^{-w}+n-3.
$$
So $P_{1}(e^{v}, e^{w})$ and $P_{n}(e^{v}, e^{w})$ only contribute to $D, D'$. When $2\leq k\leq n-1$ the coefficient of $e^{v+w}$ in $P_{k}(e^{v}, e^{w})$ is the number of admissible tuples of length $n$ with center $k$ satisfying
$$
s_{1}=0, s_{2}=1, s_{n-1}=1, s_{n}=0
$$
by above intersection number formulas. This number is $C(n-3, k-2)$. Similarly the coefficient of $e^{-v+w}$ in $P_{k}(e^{v}, e^{w})$ is the number of admissible tuples of length $n$ with center $k$ satisfying
$$
s_{1}=1, s_{2}=1, s_{n-1}=1, s_{n}=0.
$$
This number is $C(n-3, k-2)$. The third case is the coefficient of $e^{w}$ in $P_{k}(e^{v}, e^{w})$, which is the number of admissible tuples of length $n$ with center $k$ satisfying
$$
s_{2}=0, s_{n-1}=1, s_{n}=0 \quad \text{or}\quad s_{1}=1, s_{2}=2, s_{n-1}=1, s_{n}=0.
$$
When $k=1$ or $k=n$ this number is 1 as we have the explicit expressions for $P_{1}(e^{v}, e^{w})$ and $P_{k}(e^{v}, e^{w})$. When $k=2, n-1$ this number is $n-3$. When $3\leq k\leq n-2$ this number is $p_{n, k} :=C(n-3, k-3)+C(n-2, k-1)-C(n-3, k-2)$ where $C(n-3, k-3)$ is the number satisfying the first condition and $C(n-2, k-1)-C(n-3, k-2)$ is the number satisfying the second condition.

Therefore we have that
\begin{equation}
\begin{split}
A &=\sum_{k=2}^{n-1}C(n-3, k-2)(k-1)y_{1}^{-k}y_{2}^{k},\\
A' &=\sum_{k=2}^{n-1}C(n-3, k-2)k y_{1}^{1-k}y_{2}^{k-1},\\
D &=(n-3)[y_{1}^{-2}y_{2}^{2}+(n-2)y_{1}^{-n+1}y_{2}^{n-1}]\\
&+ \sum_{k=3}^{n-2}p_{n, k}(k-1)y_{1}^{-k}y_{2}^{k}+ (n-1)y_{1}^{-n}y_{2}^{n},\\
D' &=1+ (n-3)[2y_{1}^{-1}y_{2}+(n-1)y_{1}^{2-n}y_{2}^{n-2}]\\
&+ \sum_{k=3}^{n-2}p_{n, k}ky_{1}^{1-k}y_{2}^{k-1}+ ny_{1}^{1-n}y_{2}^{n-1},\\
G &= \sum_{k=2}^{n-1}C(n-3, k-2)(k-1)y_{1}^{-k}y_{2}^{k},\\
G' &=\sum_{k=2}^{n-1}C(n-3, k-2)ky_{1}^{1-k}y_{2}^{k-1}.
\end{split}
\end{equation}
In particular $A=G$ and $A'=G'$. This shows that the coefficient of the highest order term is a nontrivial element in $\mathbb{Z}[y_{1}, y_{2}, y_{1}^{-1}, y_{2}^{-1}]$. We denote it by $f_{1}(y_{1}, y_{2}, y_{1}^{-1}, y_{2}^{-1})$. Next we redo the above process but first solve for $Y$ by root formula involving $X$. By the same calculation we get a nontrivial element $f_{2}(y_{1}, y_{2}, y_{1}^{-1}, y_{2}^{-1})$. We can choose suitable $y_{1}, y_{2}$ such that all of $f_{1}, f_{2}$ are nonzero complex numbers. By a similar computation we can check that the constant term in (4.11) is nonzero for generic $y_{1}, y_{2}$.

In conclusion we get nonzero isolated solutions $\left(e^{v}, e^{w}\right)=(X, Y)$ of finite multiplicity with respect to previous choices of $y_{1}, y_{2}$.
\end{proof}

The above proof only gives us the existence of critical points but we do not know whether they are Morse or not. In particular we do not know that the big quantum cohomology with bulk deformation is semi-simple or not. Therefore we can not use the same method in Section 3 to produce quasi-morphisms. As Weiwei Wu pointed to us, one may combine the classification results of symplectic structures on 4-manifolds with the gluing method in \cite{WU} to analyze more carefully on the higher energy terms of above potential functions, especially when $X$ is not rational nor ruled. This approach will be more geometric and helpful to find new quasi-morphisms. We leave it for future study.

\section{Appendix 1: an algebraic lemma}
We provide an algebraic lemma with which we can show the existence of solutions of critical point equations.

\begin{lemma}
Let $F=(f_{1}(x,y), f_{2}(x,y))$ be a map from $\mathbb{C}^{2}$ to $\mathbb{C}^{2}$ where $f_{i}(x,y)$ are polynomials in $x,y$ with complex coefficients. Suppose that $F$ has an isolated zero in $\mathbb{C}^{2}$ then the system of equations
$$
\left\{
\begin{aligned}
f_{1}(x,y)+g_{1}(x,y;T) &=0\\
f_{2}(x,y)+g_{2}(x,y;T) &=0
\end{aligned}
\right.
$$
have solutions in $\Lambda_{0}^{2}$ near the zero of $F$ in the non-Archimedean topology. Here $g_{i}(x,y;T)$ are of the following form
$$
g_{i}(x,y;T)=\sum_{k=1}^{\infty} c_{i,k}x^{l_{i,k}}y^{j_{i,k}}T^{\lambda_{k}}
$$
where
$$
c_{i,k}\in \mathbb{C},\quad l_{i,k},j_{i,k}\in \mathbb{Z},\quad 0<\lambda_{k}\leq \lambda_{k+1}, \quad \lim_{k\rightarrow \infty}\lambda_{k}=+\infty
$$
and $\lambda_{k}$ lies in a finitely generated additive semigroup $\Sigma$ in $\mathbb{R}$.
\end{lemma}
\begin{proof}
For complex functions the multiplicity of an isolated zero is stable under small perturbations. This lemma is an analogue in the setting of Novikov ring with the non-Archimedean topology. The proof is essentially given in the weakly nondegenerate case of Theorem 10.4 in \cite{FOOO1}, where they assume that $\Sigma$ is generated by a single element such that $\Sigma$ can be rearranged to be $\mathbb{Z}_{\geq 0}$. When $\Sigma$ is finitely generated the proof is similar. We deal with the solution space of (10.8) in \cite{FOOO1} as a polynomial of several variables instead of a polynomial of one variable. And the target of the projection (10.9) in \cite{FOOO1} will be modified to a punctured polydisk. Then other parts of the proof follow similarly. We refer to Lemma 9.18, 10.12, 10.13, 10.14 and 10.15 in \cite{FOOO1} for more details.

There is also a more algebraic proof by the theory of complete local rings and the theory of Puiseux and Hahn rings. We omit it here but thank Jason Starr for helpful explanations.
\end{proof}

\section{Appendix 2: results of other semi-Fano toric surfaces}
We list the potential functions, their critical points and certain bulk deformations to obtain quasi-morphisms in the cases of $X_{3}, \cdots, X_{10}$.

\textbf{The case of $X_{3}$.} First we consider the toric surface $X_{3}$ corresponding to the following moment polytope.

$l_{1}: 0.5-\alpha_{1}+u_{1} \geq 0;$

$l_{2}: u_{2} \geq 0;$

$l_{3}: 7-u_{1}-u_{2} \geq 0;$

$l_{4}: 5-\alpha_{2}-u_{2} \geq 0;$

$l_{5}: 3-\alpha_{3}+u_{1}-u_{2} \geq 0;$

$l_{6}: 1+2u_{1}-u_{2} \geq 0.$

This moment polytope is obtained from the moment polytope of $X_{3}$ in \cite{CL} where we set $t_{1}=t_{2}=t_{3}=t_{4}=1$. When $\alpha_{i}\rightarrow 0$ we get the toric orbifold $X_{3,0}$ with one singularity of $A_{1}$-type and one singularity of $A_{2}$-type. We deform the singularities from toric resolution to the smoothing $\hat{X}_{3}$ such that all the one-point open Gromov-Witten invariants are preserved for toric fibers inside $u_{1}\geq 0, u_{2}\leq 4.8$. Then the potential function of a toric fiber $L(u)$ is
\begin{equation}
\begin{split}
\mathfrak{PO}_{\mathfrak{b}}^{L(u)}(y_{1}, y_{2}) &=(e^{a_{1}}+e^{-a_{1}+a_{2}+a_{6}})y_{1}T^{0.5+u_{1}}+ e^{a_{2}}y_{2}T^{u_{2}}+ e^{a_{3}}y_{1}^{-1}y_{2}^{-1}T^{7-u_{1}-u_{2}}\\
&+ (e^{a_{4}}+e^{a_{3}-a_{4}+a_{5}}+e^{a_{3}-a_{5}+a_{6}})y_{2}^{-1}T^{5-u_{2}}\\
&+ (e^{a_{5}}+e^{a_{4}-a_{5}+a_{6}}+e^{a_{3}-a_{4}+a_{6}})y_{1}y_{2}^{-1}T^{3+u_{1}-u_{2}}\\
&+ e^{a_{6}}y_{1}^{2}y_{2}^{-1}T^{1+2u_{1}-u_{2}}
\end{split}
\end{equation}
where $\mathfrak{b}=\sum_{i=1}^{6} a_{i}PD([D_{i}])$ is a bulk deformation.

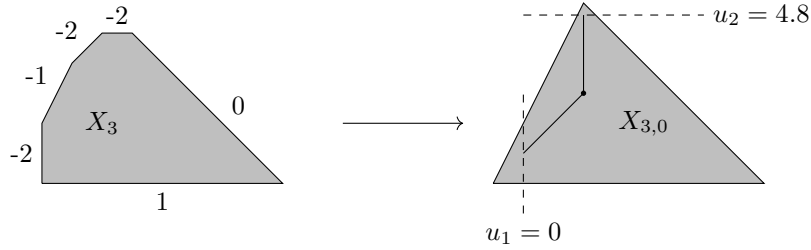
\begin{figure}
  \begin{tikzpicture}[xscale=0.8, yscale=0.8]
  \path [fill=lightgray, lightgray] (-8,0)--(-8,1)--(-7.5,2)--(-7,2.5)--(-6.5,2.5)--(-4,0);
  \draw (-8,0)--(-8,1)--(-7.5,2)--(-7,2.5)--(-6.5,2.5)--(-4,0)--(-8,0);
  \node at (-7,1) {$X_{3}$};
  \node [above] at (-6.75,2.5) {-2};
  \node [above left] at (-7.25,2.25) {-2};
  \node [left] at (-8,0.5) {-2};
  \node [below] at (-6,0) {1};
  \node [above right] at (-5,1) {0};
  \node [above left] at (-7.75, 1.5) {-1};
  \draw [->] (-3,1)--(-1,1);
  \path [fill=lightgray, lightgray] (-0.5,0)--(0,1)--(0.5,2)--(1,3)--(4,0);
  \draw (-0.5,0)--(0,1)--(0.5,2)--(1,3)--(4,0)--(-0.5,0);
  \node at (2,1) {$X_{3,0}$};
  \draw [dashed] (0,-0.5)--(0,1.5);
  \draw [dashed] (0,2.8)--(3,2.8);
  \draw (0,0.5)--(1,1.5)--(1,2.8);
  \node [right] at (3,2.8) {$u_{2}=4.8$};
  \node [below] at (0,-0.5) {$u_{1}=0$};
  \filldraw[black] (1,1.5) circle (1pt);
  \end{tikzpicture}
  \caption{Toric degeneration from $X_{3}$ to $X_{3,0}$.}
  \label{fig: Moment polytope 2}
\end{figure}

On the two segments
$$
-u_{1}+u_{2}=0.5, \quad 2.5\leq u_{2}\leq 4.8
$$
and
$$
u_{1}=2, \quad 0.5\leq u_{2}\leq 2.5
$$
the Lagrangian fiber $L(u)$ has nontrivial deformed Floer cohomology hence it is nondisplaceable.

The fiber $L(u)=L(2, 2.5)$ is a monotone Lagrangian torus with minimal Maslov index 2. It bounds 11 families of holomorphic disks and its potential function modulo energy parameter is
\begin{equation}
\begin{split}
&\mathfrak{PO}_{\mathfrak{b}}^{L(u)}(y_{1},y_{2})= (e^{a_{1}}+e^{-a_{1}+a_{2}+a_{6}})y_{1}+e^{a_{2}}y_{2}+ e^{a_{3}}y_{1}^{-1}\\
&+e^{a_{4}}y_{2}^{-1}+(e^{a_{5}}+e^{a_{4}-a_{5}+a_{6}})y_{1}y_{2}^{-1}+e^{a_{6}}y_{1}^{2}y_{2}^{-1}
\end{split}
\end{equation}
When there is no bulk deformation we can check that there are not enough nonzero critical points. Hence we consider the following bulk deformation
$$
a=(a_{1}, a_{2}, a_{3}, a_{4}, 0, 0); \quad e^{a_{1}}=e^{a_{3}}=2, e^{a_{2}}=e^{a_{4}}=-2
$$
to perturb the potential function. Then the bulk-deformed potential function is
\begin{equation}
\mathfrak{PO}_{\mathfrak{b}}^{L(u)}=y_{1}-2y_{2}+2y_{1}^{-1}y_{2}^{-1}-y_{2}^{-1}-2y_{1}y_{2}^{-1}+y_{1}^{2}y_{2}^{-1}
\end{equation}
which has 6 nondegenerate critical points. Therefore with respect to this bulk deformation the deformed Kodaira-Spencer map is an isomorphism and it gives us a quasi-morphism $\mu_{\mathfrak{b}}$.

\textbf{The case of $X_{4}$.} We consider the toric surface $X_{4}$ corresponding to the following moment polytope.

$l_{1}: 0.5-\alpha_{1}+u_{1} \geq 0;$

$l_{2}: u_{2} \geq 0;$

$l_{3}: 3.5-u_{1} \geq 0;$

$l_{4}: 4-u_{2} \geq 0;$

$l_{5}: 2.5-\alpha_{2}+u_{1}-u_{2} \geq 0;$

$l_{6}: 1+2u_{1}-u_{2} \geq 0.$

This moment polytope is obtained from the moment polytope of $X_{4}$ in \cite{CL} where we set $t_{1}=t_{3}=t_{4}=1, t_{2}=1.5$. When $\alpha_{i}\rightarrow 0$ we get the toric orbifold $X_{4,0}$ with 2 singularities of $A_{1}$-type. We deform the singularities from toric resolution to the smoothing $\hat{X}_{4}$ such that all the one-point open Gromov-Witten invariants are preserved for toric fibers inside $0\leq u_{1}, 0\leq 2+u_{1}-u_{2}$. Then the potential function of a toric fiber $L(u)$ is
\begin{equation}
\begin{split}
&\mathfrak{PO}_{\mathfrak{b}}^{L(u)}(y_{1},y_{2})= (e^{a_{1}}+e^{-a_{1}+a_{2}+a_{6}})y_{1}T^{0.5+u_{1}}+e^{a_{2}}y_{2}T^{u_{2}}+ e^{a_{3}}y_{1}^{-1}T^{3.5-u_{1}}\\
&+e^{a_{4}}y_{2}^{-1}T^{4-u_{2}}+(e^{a_{5}}+e^{a_{4}-a_{5}+a_{6}})y_{1}y_{2}^{-1}T^{2.5+u_{1}-u_{2}}+e^{a_{6}}y_{1}^{2}y_{2}^{-1}T^{1+2u_{1}-u_{2}}
\end{split}
\end{equation}
where $\mathfrak{b}=\sum_{i=1}^{6} a_{i}PD([D_{i}])$ is a bulk deformation.

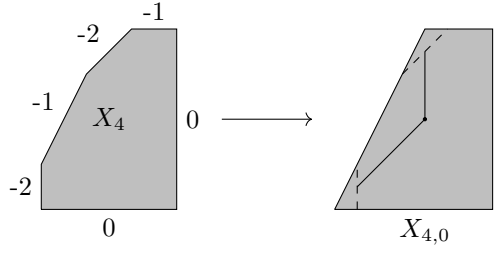
\begin{figure}
  \begin{tikzpicture}[xscale=0.6, yscale=0.6]
  \path [fill=lightgray, lightgray] (0,0)--(3,0)--(3,4)--(2,4)--(1,3)--(0,1);
  \draw (0,0)--(3,0)--(3,4)--(2,4)--(1,3)--(0,1)--(0,0);
  \node at (1.5,2) {$X_{4}$};
  \node [left] at (0,0.5) {-2};
  \node [below] at (1.5,0) {0};
  \node [right] at (3,2) {0};
  \node [above] at (2.5,4) {-1};
  \node [above left] at (1.5,3.5) {-2};
  \node [above left] at (0.5,2) {-1};
  \draw [->] (4,2)--(6,2);
  \path [fill=lightgray, lightgray] (6.5,0)--(10,0)--(10,4)--(8.5,4);
  \draw (6.5,0)--(10,0)--(10,4)--(8.5,4)--(6.5,0);
  \node [below] at (8.5,0) {$X_{4,0}$};
  \draw [dashed] (7,0)--(7,1);
  \draw [dashed] (8,3)--(9,4);
  \draw (8.5,2)--(8.5,3.5);
  \draw (7,0.5)--(8.5,2);
  \filldraw[black] (8.5,2) circle (1pt);
  \end{tikzpicture}
  \caption{Toric degeneration from $X_{4}$ to $X_{4,0}$.}
  \label{fig: Moment polytope 3}
\end{figure}

On the two segments
$$
-u_{1}+u_{2}=0.5, \quad 0\leq u_{1}\leq 1.5
$$
and
$$
u_{1}=1.5, \quad 2\leq u_{2}\leq 3.5
$$
the Lagrangian fiber $L(u)$ has nontrivial deformed Floer cohomology hence it is nondisplaceable.

The fiber $L(u)=L(1.5, 2)$ is a monotone Lagrangian torus with minimal Maslov index 2. It bounds 8 families of holomorphic disks and its potential function modulo energy parameter is
\begin{equation}
\begin{split}
&\mathfrak{PO}_{\mathfrak{b}}^{L(u)}(y_{1},y_{2})= (e^{a_{1}}+e^{-a_{1}+a_{2}+a_{6}})y_{1}+e^{a_{2}}y_{2}+ e^{a_{3}}y_{1}^{-1}\\
&+e^{a_{4}}y_{2}^{-1}+(e^{a_{5}}+e^{a_{4}-a_{5}+a_{6}})y_{1}y_{2}^{-1}+e^{a_{6}}y_{1}^{2}y_{2}^{-1}.
\end{split}
\end{equation}
When there is no bulk deformation we can check that there are not enough nonzero critical points. Hence we consider the following bulk deformation
$$
a=(a_{1}, 0, 0, 0, a_{5}, 0); \quad e^{a_{1}}+e^{-a_{1}}=3, e^{a_{5}}+e^{-a_{5}}=5
$$
to perturb the potential function. Then the bulk-deformed potential function is
\begin{equation}
\mathfrak{PO}_{\mathfrak{b}}^{L(u)}=3y_{1}+y_{2}+y_{1}^{-1}+y_{2}^{-1}+5y_{1}y_{2}^{-1}+y_{2}^{-1}+y_{1}^{2}y_{2}^{-1}
\end{equation}
which has 6 nondegenerate critical points. Therefore with respect to this bulk deformation the deformed Kodaira-Spencer map is an isomorphism and it gives us a quasi-morphism $\mu_{\mathfrak{b}}$.

\textbf{The case of $X_{5}$.} We consider the toric surface $X_{5}$ corresponding to the following moment polytope.

$l_{1}: u_{1} \geq 0;$

$l_{2}: u_{2} \geq 0;$

$l_{3}: 2-u_{1} \geq 0;$

$l_{4}: 3-u_{1}-u_{2} \geq 0;$

$l_{5}: 2-\alpha-u_{2} \geq 0;$

$l_{6}: 1+u_{1}-u_{2} \geq 0.$

This moment polytope is obtained from the moment polytope of $X_{5}$ in \cite{CL} where we set $t_{1}=t_{3}=1, t_{2}=t_{4}=0.5$. When $\alpha\rightarrow 0$ we get the toric orbifold $X_{5,0}$ with one singularity of $A_{1}$-type. We deform this singularity from toric resolution to the smoothing $\hat{X}_{5}$ such that all the one-point open Gromov-Witten invariants are preserved for toric fibers inside $u_{2}\leq 1.5$. Then the potential function of a toric fiber $L(u)$ is
\begin{equation}
\begin{split}
&\mathfrak{PO}_{\mathfrak{b}}^{L(u)}(y_{1},y_{2})= e^{a_{1}}y_{1}T^{u_{1}}+e^{a_{2}}y_{2}T^{u_{2}}+ e^{a_{3}}y_{1}^{-1}T^{2-u_{1}}\\
&+e^{a_{4}}y_{1}^{-1}y_{2}^{-1}T^{3-u_{1}-u_{2}}+(e^{a_{5}}+e^{a_{4}-a_{5}+a_{6}})y_{2}^{-1}T^{2-u_{2}}+e^{a_{6}}y_{1}y_{2}^{-1}T^{1+u_{1}-u_{2}}
\end{split}
\end{equation}
where $\mathfrak{b}=\sum_{i=1}^{6} a_{i}PD([D_{i}])$ is a bulk deformation.

\begin{figure}
  \begin{tikzpicture}
  \path [fill=lightgray, lightgray] (0,0)--(2,0)--(2,1)--(1.5,1.5)--(0.5,1.5)--(0,1);
  \draw (0,0)--(2,0)--(2,1)--(1.5,1.5)--(0.5,1.5)--(0,1)--(0,0);
  \node at (1,0.75) {$X_{5}$};
  \node [left] at (0,0.5) {-1};
  \node [below] at (1,0) {0};
  \node [right] at (2,0.5) {-1};
  \node [above] at (1,1.5) {-2};
  \node [above right] at (1.75,1.25) {-1};
  \node [above left] at (0.25,1.25) {-1};
  \draw [->] (3,1)--(5,1);
  \path [fill=lightgray, lightgray] (6,0)--(8,0)--(8,1)--(7,2)--(6,1);
  \draw (6,0)--(8,0)--(8,1)--(7,2)--(6,1)--(6,0);
  \node [below] at (7,0) {$X_{5,0}$};
  \draw [dashed] (6.5,1.5)--(7.5,1.5);
  \draw (7,1.5)--(7,1);
  \filldraw[black] (7,1) circle (1pt);
  \end{tikzpicture}
  \caption{Toric degeneration from $X_{5}$ to $X_{5,0}$.}
  \label{fig: Moment polytope 4}
\end{figure}
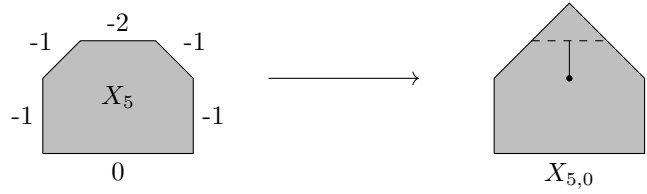

On the segment
$$
u_{1}=1, \quad 1\leq u_{2}\leq 1.5
$$
the Lagrangian fiber $L(u)$ has nontrivial deformed Floer cohomology hence it is nondisplaceable.

The fiber $L(u)=L(1, 1)$ is a monotone Lagrangian torus with minimal Maslov index 2. It bounds 7 families of holomorphic disks and its potential function modulo energy parameter is
\begin{equation}
\begin{split}
&\mathfrak{PO}_{\mathfrak{b}}^{L(u)}(y_{1},y_{2})= e^{a_{1}}y_{1}+e^{a_{2}}y_{2}+ e^{a_{3}}y_{1}^{-1}\\
&+e^{a_{4}}y_{1}^{-1}y_{2}^{-1}+(e^{a_{5}}+e^{a_{4}-a_{5}+a_{6}})y_{2}^{-1}+e^{a_{6}}y_{1}y_{2}^{-1}
\end{split}
\end{equation}
When there is no bulk deformation we can check that there exist critical points but not of maximal number. Hence we consider the following bulk deformation
$$
a=(a_{1}, 0, 0, 0, a_{5}, 0); \quad e^{a_{1}}=2, e^{a_{5}}+e^{-a_{5}}=3
$$
to perturb the potential function. Then the bulk-deformed potential function is
\begin{equation}
\mathfrak{PO}_{\mathfrak{b}}^{L(u)}=2y_{1}+y_{2}+y_{1}^{-1}+y_{1}^{-1}y_{2}^{-1}+3y_{2}^{-1}+y_{1}y_{2}^{-1}
\end{equation}
which has 6 nondegenerate critical points. Therefore with respect to this bulk deformation the deformed Kodaira-Spencer map is an isomorphism and it gives us a quasi-morphism $\mu_{\mathfrak{b}}$.

\textbf{The case of $X_{6}$.} We consider the toric surface $X_{6}$ corresponding to the following moment polytope.

$l_{1}: 0.5-\alpha_{1}+u_{1} \geq 0;$

$l_{2}: u_{2} \geq 0;$

$l_{3}: 8.5-u_{1} \geq 0;$

$l_{4}: 13-u_{1}-u_{2} \geq 0;$

$l_{5}: 9-\alpha_{2}-u_{2} \geq 0;$

$l_{6}: 5-\alpha_{3}+u_{1}-u_{2} \geq 0;$

$l_{7}: 1+2u_{1}-u_{2} \geq 0.$

This moment polytope is obtained from the moment polytope of $X_{6}$ in \cite{CL} where we set $t_{1}=t_{5}=1, t_{2}=1.5, t_{3}=t_{4}=3$. When $\alpha\rightarrow 0$ we get the toric orbifold $X_{6,0}$ with one singularity of $A_{2}$-type at the top and one singularity of $A_{1}$-type. We deform these singularities from toric resolution to the smoothing $\hat{X}_{6}$ such that all the one-point open Gromov-Witten invariants are preserved for toric fibers inside $0\leq u_{1}, u_{2}\leq 8$. Then the potential function of a toric fiber $L(u)$ is
\begin{equation}
\begin{split}
&\mathfrak{PO}_{\mathfrak{b}}^{L(u)}(y_{1},y_{2})= (e^{a_{1}}+e^{-a_{1}+a_{2}+a_{7}})y_{1}T^{0.5+u_{1}}+e^{a_{2}}y_{2}T^{u_{2}}+ e^{a_{3}}y_{1}^{-1}T^{8.5-u_{1}}\\
&+e^{a_{4}}y_{1}^{-1}y_{2}^{-1}T^{13-u_{1}-u_{2}}+(e^{a_{5}}+e^{a_{4}-a_{5}+a_{6}}+e^{a_{4}-a_{6}+a_{7}})y_{2}^{-1}T^{9-u_{2}}\\
&+(e^{a_{6}}+e^{a_{5}-a_{6}+a_{7}}+e^{a_{4}-a_{5}+a_{7}})y_{1}y_{2}^{-1}T^{5+u_{1}-u_{2}}+e^{a_{7}}y_{1}^{2}y_{2}^{-1}T^{1+2u_{1}-u_{2}}
\end{split}
\end{equation}
where $\mathfrak{b}=\sum_{i=1}^{7} a_{i}PD([D_{i}])$ is a bulk deformation.

\begin{figure}
  \begin{tikzpicture}[xscale=0.5, yscale=0.4]
  \path [fill=lightgray, lightgray] (0,0)--(7.5,0)--(7.5,5.5)--(7,6)--(4,6)--(1,3)--(0,1);
  \draw (0,0)--(7.5,0)--(7.5,5.5)--(7,6)--(4,6)--(1,3)--(0,1)--(0,0);
  \node at (4,2) {$X_{6}$};
  \node [left] at (0,0.5) {-2};
  \node [below] at (4,0) {0};
  \node [right] at (7.5,2.5) {-1};
  \node [above] at (5.5,6) {-2};
  \node [above right] at (7.25,5.75) {-1};
  \node [above left] at (2.5,4.5) {-2};
  \node [above left] at (0.5,2) {-1};
  \draw [->] (9,2)--(11,2);
  \path [fill=lightgray, lightgray] (11.5,0)--(19.5,0)--(19.5,5.5)--(16,9);
  \draw (11.5,0)--(19.5,0)--(19.5,5.5)--(16,9)--(11.5,0);
  \node [below] at (16,0) {$X_{6,0}$};
  \draw [dashed] (12,0)--(12,1);
  \draw [dashed] (15.5,8)--(17,8);
  \draw (16,8)--(16,4.5)--(12,0.5);
  \filldraw[black] (16,4.5) circle (1pt);
  \end{tikzpicture}
  \caption{Toric degeneration from $X_{6}$ to $X_{6,0}$.}
  \label{fig: Moment polytope 5}
\end{figure}
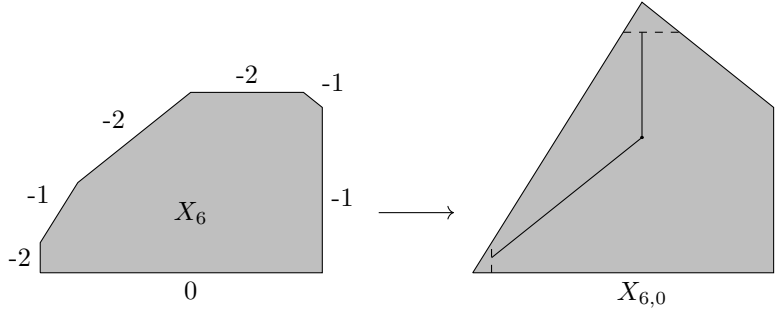

On the two segments
$$
u_{1}=4, \quad 4.5\leq u_{2}\leq 8
$$
and
$$
-u_{1}+u_{2}=0.5, \quad 0\leq u_{1}\leq 4
$$
the Lagrangian fiber $L(u)$ has nontrivial deformed Floer cohomology hence it is nondisplaceable.

The fiber $L(u)=L(4, 4.5)$ at the barycenter is a monotone Lagrangian torus with minimal Maslov index 2. It bounds 12 families of holomorphic disks and its potential function modulo energy parameter is \begin{equation}
\begin{split}
&\mathfrak{PO}_{\mathfrak{b}}^{L(u)}(y_{1},y_{2})= (e^{a_{1}}+e^{-a_{1}+a_{2}+a_{7}})y_{1}+e^{a_{2}}y_{2}+ e^{a_{3}}y_{1}^{-1}\\
&+e^{a_{4}}y_{1}^{-1}y_{2}^{-1}+(e^{a_{5}}+e^{a_{4}-a_{5}+a_{6}}+e^{a_{4}-a_{6}+a_{7}})y_{2}^{-1}\\
&+(e^{a_{6}}+e^{a_{5}-a_{6}+a_{7}}+e^{a_{4}-a_{5}+a_{7}})y_{1}y_{2}^{-1}+e^{a_{7}}y_{1}^{2}y_{2}^{-1}
\end{split}
\end{equation}
When there is no bulk deformation we can check that there exist critical points but not of maximal number. Hence we consider the following bulk deformation
$$
a=(a_{1}, 0, a_{3}, 0, a_{5}, a_{6}, 0)
$$
such that
$$
e^{a_{1}}+e^{-a_{1}}=3, \quad e^{a_{3}}=3, \quad  e^{a_{5}}=2, \quad e^{a_{6}}=-2
$$
to perturb the potential function. Then the bulk-deformed potential function is
\begin{equation}
\mathfrak{PO}_{\mathfrak{b}}^{L(u)}=3y_{1}+y_{2}+3y_{1}^{-1}+y_{1}^{-1}y_{2}^{-1}+0.5y_{2}^{-1}-2.5y_{1}y_{2}^{-1}+y_{1}^{2}y_{2}^{-1}
\end{equation}
which has 7 nondegenerate critical points. Therefore with respect to this bulk deformation the deformed Kodaira-Spencer map is an isomorphism and it gives us a quasi-morphism $\mu_{\mathfrak{b}}$.

\textbf{The case of $X_{7}$.} We consider the toric surface $X_{7}$ corresponding to the following moment polytope.

$l_{1}: 0.5-\alpha_{1}+u_{1} \geq 0;$

$l_{2}: u_{2} \geq 0;$

$l_{3}: 1.5-u_{1}+u_{2} \geq 0;$

$l_{4}: 3.5-u_{1} \geq 0;$

$l_{5}: 4-u_{2} \geq 0;$

$l_{6}: 2.5-\alpha_{2}+u_{1}-u_{2} \geq 0;$

$l_{7}: 1+2u_{1}-u_{2} \geq 0.$

This moment polytope is obtained from the moment polytope of $X_{7}$ in \cite{CL} where we set $t_{1}=t_{4}=t_{5}=1, t_{2}=2, t_{3}=1.5$. When $\alpha\rightarrow 0$ we get the toric orbifold $X_{7,0}$ with 2 singularities of $A_{2}$-type. We deform these singularities from toric resolution to the smoothing $\hat{X}_{7}$ such that all the one-point open Gromov-Witten invariants are preserved for toric fibers inside $0\leq u_{1}, 0\leq 2+u_{1}-u_{2}$. Then the potential function of a toric fiber $L(u)$ is
\begin{equation}
\begin{split}
&\mathfrak{PO}_{\mathfrak{b}}^{L(u)}(y_{1},y_{2})= (e^{a_{1}}+e^{-a_{1}+a_{2}+a_{7}})y_{1}T^{0.5+u_{1}}+e^{a_{2}}y_{2}T^{u_{2}}+ e^{a_{3}}y_{1}^{-1}y_{2}T^{1.5-u_{1}+u_{2}}\\
&+e^{a_{4}}y_{1}^{-1}T^{3.5-u_{1}}+e^{a_{5}}y_{2}^{-1}T^{4-u_{2}}\\
&+(e^{a_{6}}+e^{a_{5}-a_{6}+a_{7}})y_{1}y_{2}^{-1}T^{2.5+u_{1}-u_{2}}+e^{a_{7}}y_{1}^{2}y_{2}^{-1}T^{1+2u_{1}-u_{2}}
\end{split}
\end{equation}
where $\mathfrak{b}=\sum_{i=1}^{7} a_{i}PD([D_{i}])$ is a bulk deformation.

\begin{figure}
  \begin{tikzpicture}[xscale=0.8, yscale=0.8]
  \path [fill=lightgray, lightgray] (0,0)--(1.5,0)--(3.5,2)--(3.5,4)--(2,4)--(1,3)--(0,1);
  \draw (0,0)--(1.5,0)--(3.5,2)--(3.5,4)--(2,4)--(1,3)--(0,1)--(0,0);
  \node at (1.5,2) {$X_{7}$};
  \node [left] at (0,0.5) {-2};
  \node [below] at (0.75,0) {-1};
  \node [below right] at (2.5,1) {-1};
  \node [right] at (3.5,3) {-1};
  \node [above] at (2.75,4) {-1};
  \node [above right] at (1,3.5) {-2};
  \node [above left] at (0.5,2) {-1};
  \draw [->] (4,2)--(6,2);
  \path [fill=lightgray, lightgray] (6.5,0)--(8.5,0)--(10.5,2)--(10.5,4)--(8.5,4);
  \draw (6.5,0)--(8.5,0)--(10.5,2)--(10.5,4)--(8.5,4)--(6.5,0);
  \node [below] at (7.5,0) {$X_{7,0}$};
  \draw [dashed] (7,0)--(7,1);
  \draw [dashed] (8,3)--(9,4);
  \draw (8.5,3.5)--(8.5,2)--(7,0.5);
  \filldraw[black] (8.5,2) circle (1pt);
  \end{tikzpicture}
  \caption{Toric degeneration from $X_{7}$ to $X_{7,0}$.}
  \label{fig: Moment polytope 6}
\end{figure}
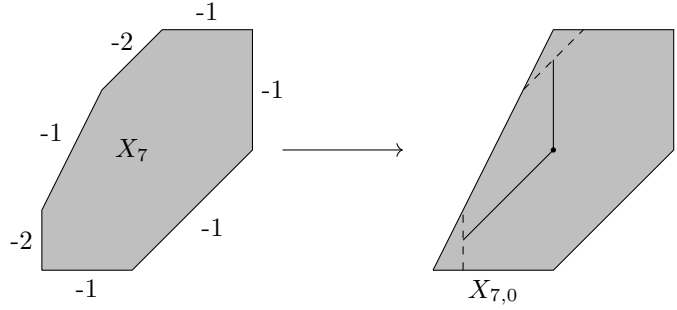

On the two segments
$$
u_{1}=1.5, \quad 2\leq u_{2}\leq 3.5
$$
and
$$
-u_{1}+u_{2}=0.5, \quad 0\leq u_{1}\leq 1.5
$$
the Lagrangian fiber $L(u)$ has nontrivial deformed Floer cohomology hence it is nondisplaceable.

The fiber $L(u)=L(1.5, 2)$ at the barycenter is a monotone Lagrangian torus with minimal Maslov index 2. It bounds 9 families of holomorphic disks and its potential function modulo energy parameter is \begin{equation}
\begin{split}
&\mathfrak{PO}_{\mathfrak{b}}^{L(u)}(y_{1},y_{2})= (e^{a_{1}}+e^{-a_{1}+a_{2}+a_{7}})y_{1}+e^{a_{2}}y_{2}+ e^{a_{3}}y_{1}^{-1}y_{2}\\
&+e^{a_{4}}y_{1}^{-1}+e^{a_{5}}y_{2}^{-1}\\
&+(e^{a_{6}}+e^{a_{5}-a_{6}+a_{7}})y_{1}y_{2}^{-1}+e^{a_{7}}y_{1}^{2}y_{2}^{-1}
\end{split}
\end{equation}
When there is no bulk deformation we can check that there exist critical points but not of maximal number. Hence we consider the following bulk deformation
$$
a=(a_{1}, 0, a_{3}, 0, 0, a_{6}, 0)
$$
such that
$$
e^{a_{1}}+e^{-a_{1}}=3, \quad e^{a_{3}}=2, \quad e^{a_{6}}+e^{-a_{6}}=4
$$
to perturb the potential function. Then the bulk-deformed potential function is
\begin{equation}
\mathfrak{PO}_{\mathfrak{b}}^{L(u)}=3y_{1}+y_{2}+2y_{1}^{-1}y_{2}+y_{1}^{-1}+y_{2}^{-1}+4y_{1}y_{2}^{-1}+y_{1}^{2}y_{2}^{-1}
\end{equation}
which has 7 nondegenerate critical points. Therefore with respect to this bulk deformation the deformed Kodaira-Spencer map is an isomorphism and it gives us a quasi-morphism $\mu_{\mathfrak{b}}$.

\textbf{The case of $X_{8}$.} We consider the toric surface $X_{8}$ corresponding to the following moment polytope.

$l_{1}: 0.5-\alpha_{1}+u_{1} \geq 0;$

$l_{2}: u_{2} \geq 0;$

$l_{3}: 5.5-\alpha_{2}-u_{1} \geq 0;$

$l_{4}: 11-2u_{1}-u_{2} \geq 0;$

$l_{5}: 8.5-\alpha_{3}-u_{1}-u_{2} \geq 0;$

$l_{6}: 6-\alpha_{4}-u_{2} \geq 0;$

$l_{7}: 3.5-\alpha_{5}+u_{1}-u_{2} \geq 0;$

$l_{8}: 1+2u_{1}-u_{2} \geq 0.$

This moment polytope is obtained from the moment polytope of $X_{8}$ in \cite{CL} where we set $t_{1}=\cdots=t_{6}=1$. When $\alpha\rightarrow 0$ we get the toric orbifold $X_{9,0}$ with one singularity of $A_{3}$-type and two singularities of $A_{1}$-type. We deform these singularities from toric resolution to the smoothing $\hat{X}_{8}$ such that all the one-point open Gromov-Witten invariants are preserved for toric fibers inside $0\leq u_{1}\leq 5, u_{2}\leq 5.8$. Then the potential function of a toric fiber $L(u)$ is
\begin{equation}
\begin{split}
& \mathfrak{PO}_{\mathfrak{b}}^{L(u)}(y_{1}, y_{2}) =(e^{a_{1}}+e^{-a_{1}+a_{2}+a_{8}})y_{1}T^{0.5+u_{1}}+ e^{a_{2}}y_{2}T^{u_{2}}\\
&+ (e^{a_{3}}+e^{a_{2}-a_{3}+a_{4}})y_{1}^{-1}T^{5.5-u_{1}}+e^{a_{4}}y_{1}^{-2}y_{2}^{-1}T^{11-2u_{1}-u_{2}}\\
&+ (e^{a_{5}}+e^{a_{4}-a_{5}+a_{6}}+e^{a_{4}-a_{6}+a_{7}}+e^{a_{4}-a_{7}+a_{8}})y_{1}^{-1}y_{2}^{-1}T^{8.5-u_{1}-u_{2}}\\
&+ (e^{a_{6}}+e^{a_{5}-a_{6}+a_{7}}+e^{a_{5}-a_{7}+a_{8}}+e^{a_{4}-a_{5}+a_{7}}\\
&+ e^{a_{4}-a_{5}+a_{6}-a_{7}+a_{8}}+e^{a_{4}-a_{6}+a_{8}})y_{2}^{-1}T^{6-u_{2}}\\
&+ (e^{a_{7}}+e^{a_{6}-a_{7}+a_{8}}+e^{a_{5}-a_{6}+a_{8}}+e^{a_{4}-a_{5}+a_{8}})y_{1}y_{2}^{-1}T^{3.5+u_{1}-u_{2}}\\
&+ e^{a_{8}}y_{1}^{2}y_{2}^{-1}T^{1+2u_{1}-u_{2}}.
\end{split}
\end{equation}
where $\mathfrak{b}=\sum_{i=1}^{8} a_{i}PD([D_{i}])$ is a bulk deformation.

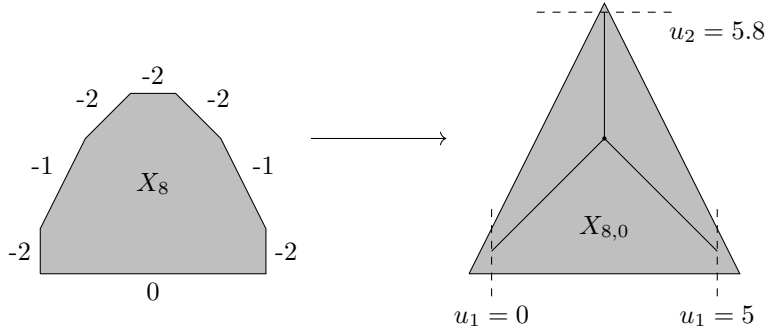
\begin{figure}
  \begin{tikzpicture}[xscale=0.6, yscale=0.6]
  \path [fill=lightgray, lightgray] (-8,0)--(-8,1)--(-7,3)--(-6,4)--(-5,4)--(-4,3)--(-3,1)--(-3,0);
  \draw (-8,0)--(-8,1)--(-7,3)--(-6,4)--(-5,4)--(-4,3)--(-3,1)--(-3,0)--(-8,0);
  \node at (-5.5,2) {$X_{8}$};
  \node [left] at (-8,0.5) {-2};
  \node [below] at (-5.5,0) {0};
  \node [right] at (-3,0.5) {-2};
  \node [above right] at (-3.5,2) {-1};
  \node [above right] at (-4.5,3.5) {-2};
  \node [above] at (-5.5,4) {-2};
  \node [above left] at (-6.5,3.5) {-2};
  \node [above left] at (-7.5,2) {-1};
  \draw [->] (-2,3)--(1,3);
  \path [fill=lightgray, lightgray] (1.5,0)--(4.5,6)--(7.5,0);
  \draw (1.5,0)--(4.5,6)--(7.5,0)--(1.5,0);
  \node [below] at (4.5,1.5) {$X_{8,0}$};
  \draw [dashed] (2,-0.5)--(2,1.5);
  \draw [dashed] (3,5.8)--(6,5.8);
  \draw [dashed] (7,-0.5)--(7,1.5);
  \draw (2,0.5)--(4.5,3)--(4.5,5.8);
  \draw (4.5,3)--(7,0.5);
  \node [below] at (2,-0.5) {$u_{1}=0$};
  \node [below] at (7,-0.5) {$u_{1}=5$};
  \node [below] at (7,5.8) {$u_{2}=5.8$};
  \filldraw[black] (4.5,3) circle (1pt);
  \end{tikzpicture}
  \caption{Toric degeneration from $X_{8}$ to $X_{8,0}$.}
  \label{fig: Moment polytope 7}
\end{figure}

On the three segments
$$
u_{1}=2.5, \quad 3\leq u_{2}\leq 5.8
$$
and
$$
-u_{1}+u_{2}=0.5, \quad 0\leq u_{1}\leq 2.5
$$
and
$$
u_{1}+u_{2}=5.5, \quad 2.5\leq u_{1}\leq 5
$$
the Lagrangian fiber $L(u)$ has nontrivial deformed Floer cohomology hence it is nondisplaceable.

The fiber $L(u)=L(2.5, 3)$ at the barycenter is a monotone Lagrangian torus with minimal Maslov index 2. It bounds 21 families of holomorphic disks and its potential function modulo energy parameter is \begin{equation}
\begin{split}
& \mathfrak{PO}_{\mathfrak{b}}^{L(u)}(y_{1}, y_{2}) =(e^{a_{1}}+e^{-a_{1}+a_{2}+a_{8}})y_{1}+ e^{a_{2}}y_{2}\\
&+ (e^{a_{3}}+e^{a_{2}-a_{3}+a_{4}})y_{1}^{-1}+e^{a_{4}}y_{1}^{-2}y_{2}^{-1}\\
&+ (e^{a_{5}}+e^{a_{4}-a_{5}+a_{6}}+e^{a_{4}-a_{6}+a_{7}}+e^{a_{4}-a_{7}+a_{8}})y_{1}^{-1}y_{2}^{-1}\\
&+ (e^{a_{6}}+e^{a_{5}-a_{6}+a_{7}}+e^{a_{5}-a_{7}+a_{8}}+e^{a_{4}-a_{5}+a_{7}}\\
&+ e^{a_{4}-a_{5}+a_{6}-a_{7}+a_{8}}+e^{a_{4}-a_{6}+a_{8}})y_{2}^{-1}\\
&+ (e^{a_{7}}+e^{a_{6}-a_{7}+a_{8}}+e^{a_{5}-a_{6}+a_{8}}+e^{a_{4}-a_{5}+a_{8}})y_{1}y_{2}^{-1}\\
&+ e^{a_{8}}y_{1}^{2}y_{2}^{-1}.
\end{split}
\end{equation}
When there is no bulk deformation we can check that there exist critical points but not of maximal number. Hence we consider the following bulk deformation
$$
a=(a_{1}, 0, a_{3}, a_{4}, a_{5}, a_{6}, 0, 0)
$$
such that
$$
e^{a_{1}}=e^{a_{6}}=2, \quad e^{a_{3}}=e^{a_{5}}=-2, \quad e^{a_{4}}=3
$$
to perturb the potential function. Then the bulk-deformed potential function is
\begin{equation}
\mathfrak{PO}_{\mathfrak{b}}^{L(u)}=2.5y_{1}+ y_{2}-3.5y_{1}^{-1}+3y_{1}^{-2}y_{2}^{-1}-0.5y_{1}^{-1}y_{2}^{-1}-4y_{2}^{-1}+0.5y_{1}y_{2}^{-1}+ y_{1}^{2}y_{2}^{-1}
\end{equation}
which has 8 nondegenerate critical points. Therefore with respect to this bulk deformation the deformed Kodaira-Spencer map is an isomorphism and it gives us a quasi-morphism $\mu_{\mathfrak{b}}$.

\textbf{The case of $X_{9}$.} We consider the toric surface $X_{9}$ corresponding to the following moment polytope.

$l_{1}: 0.5-\alpha_{1}+u_{1} \geq 0;$

$l_{2}: u_{2} \geq 0;$

$l_{3}: 4-u_{1}+u_{2} \geq 0;$

$l_{4}: 8.5-\alpha_{2}-u_{1} \geq 0;$

$l_{5}: 13-u_{1}-u_{2} \geq 0;$

$l_{6}: 9-\alpha_{3}-u_{2} \geq 0;$

$l_{7}: 5-\alpha_{4}+u_{1}-u_{2} \geq 0;$

$l_{8}: 1+2u_{1}-u_{2} \geq 0.$

This moment polytope is obtained from the moment polytope of $X_{9}$ in \cite{CL} where we set $t_{1}=t_{2}=t_{3}=t_{6}=1, t_{4}=t_{5}=3$. When $\alpha\rightarrow 0$ we get the toric orbifold $X_{9,0}$ with 2 singularities of $A_{2}$-type. We deform these singularities from toric resolution to the smoothing $\hat{X}_{9}$ such that all the one-point open Gromov-Witten invariants are preserved for toric fibers inside $0\leq u_{1}\leq 8, u_{2}\leq 8$. Then the potential function of a toric fiber $L(u)$ is
\begin{equation}
\begin{split}
&\mathfrak{PO}_{\mathfrak{b}}^{L(u)}(y_{1},y_{2})= (e^{a_{1}}+e^{-a_{1}+a_{2}+a_{8}})y_{1}T^{0.5+u_{1}}+e^{a_{2}}y_{2}T^{u_{2}}+ e^{a_{3}}y_{1}^{-1}y_{2}T^{4-u_{1}+u_{2}}\\
&+(e^{a_{4}}+e^{a_{3}-a_{4}+a_{5}})y_{1}^{-1}T^{8.5-u_{1}}+e^{a_{5}}y_{1}^{-1}y_{2}^{-1}T^{13-u_{1}-u_{2}}\\
&+(e^{a_{6}}+e^{a_{5}-a_{6}+a_{7}}+e^{a_{5}-a_{7}+a_{8}})y_{2}^{-1}T^{9-u_{2}}\\
&+(e^{a_{7}}+e^{a_{6}-a_{7}+a_{8}}+e^{a_{5}-a_{6}+a_{8}})y_{1}y_{2}^{-1}T^{5+u_{1}-u_{2}}+e^{a_{8}}y_{1}^{2}y_{2}^{-1}T^{1+2u_{1}-u_{2}}
\end{split}
\end{equation}
where $\mathfrak{b}=\sum_{i=1}^{8} a_{i}PD([D_{i}])$ is a bulk deformation.

\begin{figure}
  \begin{tikzpicture}[xscale=0.4, yscale=0.4]
  \path [fill=lightgray, lightgray] (0,0)--(4,0)--(8,4)--(8,5)--(7,6)--(4,6)--(1,3)--(0,1);
  \draw (0,0)--(4,0)--(8,4)--(8,5)--(7,6)--(4,6)--(1,3)--(0,1)--(0,0);
  \node at (4,2.5) {$X_{9}$};
  \node [left] at (0,0.5) {-2};
  \node [below] at (2,0) {-1};
  \node [below right] at (6,2) {-1};
  \node [right] at (8,4.5) {-2};
  \node [above right] at (7.5,5.5) {-1};
  \node [above] at (5.5,6) {-2};
  \node [above left] at (2.5,4.5) {-2};
  \node [above left] at (0.5,2) {-1};
  \draw [->] (9,4)--(11,4);
  \path [fill=lightgray, lightgray] (11.5,0)--(16,0)--(20.5,4.5)--(16,9);
  \draw (11.5,0)--(16,0)--(20.5,4.5)--(16,9)--(11.5,0);
  \node [below] at (14,0) {$X_{9,0}$};
  \draw [dashed] (12,0)--(12,1);
  \draw [dashed] (15.5,8)--(17,8);
  \draw [dashed] (20,4)--(20,5);
  \draw (16,8)--(16,4.5)--(12,0.5);
  \draw (16,4.5)--(20,4.5);
  \filldraw[black] (16,4.5) circle (1pt);
  \end{tikzpicture}
  \caption{Toric degeneration from $X_{9}$ to $X_{9,0}$.}
  \label{fig: Moment polytope 8}
\end{figure}
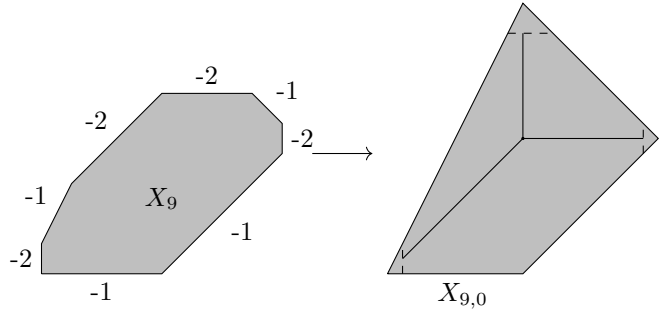

On the three segments
$$
u_{1}=4, \quad 4.5\leq u_{2}\leq 8
$$
and
$$
-u_{1}+u_{2}=0.5, \quad 0\leq u_{1}\leq 4
$$
and
$$
u_{2}=4.5, \quad 4\leq u_{1}\leq 8
$$
the Lagrangian fiber $L(u)$ has nontrivial deformed Floer cohomology hence it is nondisplaceable.

The fiber $L(u)=L(4, 4.5)$ at the barycenter is a monotone Lagrangian torus with minimal Maslov index 2. It bounds 14 families of holomorphic disks and its potential function modulo energy parameter is \begin{equation}
\begin{split}
&\mathfrak{PO}_{\mathfrak{b}}^{L(u)}(y_{1},y_{2})= (e^{a_{1}}+e^{-a_{1}+a_{2}+a_{8}})y_{1}+e^{a_{2}}y_{2}+ e^{a_{3}}y_{1}^{-1}y_{2}\\
&+(e^{a_{4}}+e^{a_{3}-a_{4}+a_{5}})y_{1}^{-1}+e^{a_{5}}y_{1}^{-1}y_{2}^{-1}\\
&+(e^{a_{6}}+e^{a_{5}-a_{6}+a_{7}}+e^{a_{5}-a_{7}+a_{8}})y_{2}^{-1}\\
&+(e^{a_{7}}+e^{a_{6}-a_{7}+a_{8}}+e^{a_{5}-a_{6}+a_{8}})y_{1}y_{2}^{-1}+e^{a_{8}}y_{1}^{2}y_{2}^{-1}
\end{split}
\end{equation}
When there is no bulk deformation we can check that there exist critical points but not of maximal number. Hence we consider the following bulk deformation
$$
a=(a_{1}, a_{2}, 0, a_{4}, 0, a_{6}, a_{7}, 0)
$$
such that
$$
e^{a_{1}}+e^{-a_{1}+a_{2}}=3, \quad e^{a_{2}}=4, \quad e^{a_{4}}+e^{-a_{4}}=4, \quad e^{a_{6}}=2, \quad e^{a_{7}}=-2
$$
to perturb the potential function. Then the bulk-deformed potential function is
\begin{equation}
\mathfrak{PO}_{\mathfrak{b}}^{L(u)}=3y_{1}+4y_{2}+y_{1}^{-1}y_{2}+4y_{1}^{-1}+y_{1}^{-1}y_{2}^{-1}+0.5y_{2}^{-1}-2.5y_{1}y_{2}^{-1}+y_{1}^{2}y_{2}^{-1}
\end{equation}
which has 8 nondegenerate critical points. Therefore with respect to this bulk deformation the deformed Kodaira-Spencer map is an isomorphism and it gives us a quasi-morphism $\mu_{\mathfrak{b}}$.

\textbf{The case of $X_{10}$.} We consider the toric surface $X_{10}$ corresponding to the following moment polytope.

$l_{1}: 0.5-\alpha_{1}+u_{1} \geq 0;$

$l_{2}: u_{2} \geq 0;$

$l_{3}: 1.5-\alpha_{2}-u_{1}+u_{2} \geq 0;$

$l_{4}: 3-2u_{1}+u_{2} \geq 0;$

$l_{5}: 3.5-\alpha_{3}-u_{1} \geq 0;$

$l_{6}: 4-u_{2} \geq 0;$

$l_{7}: 2.5-\alpha_{4}+u_{1}-u_{2} \geq 0;$

$l_{8}: 1+2u_{1}-u_{2} \geq 0.$

When $\alpha_{i}\rightarrow 0$ we get the toric orbifold $X_{10,0}$ with 4 singularities of $A_{1}$-type. We deform the singularities from toric resolution to the smoothing $\hat{X}_{10}$ such that all the one-point open Gromov-Witten invariants are preserved for toric fibers inside $0\leq u_{1}\leq 3, -1\leq -u_{1}+u_{2}\leq 2$. Then the potential function of a toric fiber $L(u)$ is
\begin{equation}
\begin{split}
&\mathfrak{PO}_{\mathfrak{b}}^{L(u)}(y_{1},y_{2})= (e^{a_{1}}+e^{-a_{1}+a_{2}+a_{8}})y_{1}T^{0.5+u_{1}}+e^{a_{2}}y_{2}T^{u_{2}}\\
&+ (e^{a_{3}}+e^{a_{2}-a_{3}+a_{4}})y_{1}^{-1}y_{2}T^{1.5-u_{1}+u_{2}}+e^{a_{4}}y_{1}^{-2}y_{2}T^{3-2u_{1}+u_{2}}\\
&+ (e^{a_{5}}+e^{a_{4}-a_{5}+a_{6}})y_{1}^{-1}T^{3.5-u_{1}}+e^{a_{6}}y_{2}^{-1}T^{4-u_{2}}\\
&+ (e^{a_{7}}+e^{a_{6}-a_{7}+a_{8}})y_{1}y_{2}^{-1}T^{2.5+u_{1}-u_{2}}+e^{a_{8}}y_{1}^{2}y_{2}^{-1}T^{1+2u_{1}-u_{2}}
\end{split}
\end{equation}
where $\mathfrak{b}=\sum_{i=1}^{8} a_{i}PD([D_{i}])$ is a bulk deformation.

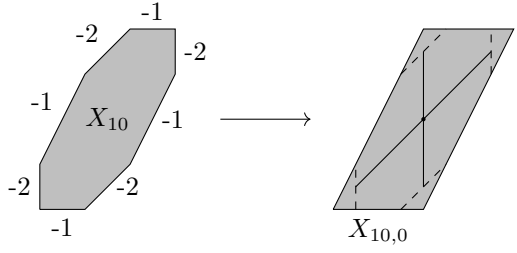
\begin{figure}
  \begin{tikzpicture}[xscale=0.6, yscale=0.6]
  \path [fill=lightgray, lightgray] (0,0)--(1,0)--(2,1)--(3,3)--(3,4)--(2,4)--(1,3)--(0,1);
  \draw (0,0)--(1,0)--(2,1)--(3,3)--(3,4)--(2,4)--(1,3)--(0,1)--(0,0);
  \node at (1.5,2) {$X_{10}$};
  \node [left] at (0,0.5) {-2};
  \node [below] at (0.5,0) {-1};
  \node [right] at (1.5,0.5) {-2};
  \node [right] at (2.5,2) {-1};
  \node [right] at (3,3.5) {-2};
  \node [above] at (2.5,4) {-1};
  \node [above left] at (1.5,3.5) {-2};
  \node [above left] at (0.5,2) {-1};
  \draw [->] (4,2)--(6,2);
  \path [fill=lightgray, lightgray] (6.5,0)--(8.5,0)--(10.5,4)--(8.5,4);
  \draw (6.5,0)--(8.5,0)--(10.5,4)--(8.5,4)--(6.5,0);
  \node [below] at (7.5,0) {$X_{10,0}$};
  \draw [dashed] (8,0)--(9,1);
  \draw [dashed] (7,0)--(7,1);
  \draw [dashed] (10,3)--(10,4);
  \draw [dashed] (8,3)--(9,4);
  \draw (8.5,0.5)--(8.5,3.5);
  \draw (7,0.5)--(10,3.5);
  \filldraw[black] (8.5,2) circle (1pt);
  \end{tikzpicture}
  \caption{Toric degeneration from $X_{10}$ to $X_{10,0}$.}
  \label{fig: Moment polytope 9}
\end{figure}

On the two segments
$$
0.5+u_{1}-u_{2}=0, \quad 0\leq u_{1}\leq 3
$$
and
$$
u_{1}=1.5, \quad 0.5\leq u_{2}\leq 3.5
$$
the Lagrangian fiber $L(u)$ has nontrivial deformed Floer cohomology hence it is nondisplaceable.

The fiber $L(u)=L(1.5, 2)$ at the barycenter is a monotone Lagrangian torus with minimal Maslov index 2. It bounds 12 families of holomorphic disks and its potential function modulo energy parameter is \begin{equation}
\begin{split}
&\mathfrak{PO}_{\mathfrak{b}}^{L(u)}(y_{1},y_{2})=(e^{a_{1}}+e^{-a_{1}+a_{2}+a_{8}})y_{1}+e^{a_{2}}y_{2}\\
&+ (e^{a_{3}}+e^{a_{2}-a_{3}+a_{4}})y_{1}^{-1}y_{2}+e^{a_{4}}y_{1}^{-2}y_{2}\\
&+ (e^{a_{5}}+e^{a_{4}-a_{5}+a_{6}})y_{1}^{-1}+e^{a_{6}}y_{2}^{-1}\\
&+ (e^{a_{7}}+e^{a_{6}-a_{7}+a_{8}})y_{1}y_{2}^{-1}+e^{a_{8}}y_{1}^{2}y_{2}^{-1}.
\end{split}
\end{equation}
When there is no bulk deformation we can check that there exist critical points but not of maximal number. Hence we consider the following bulk deformation
$$
a=(a_{1}, 0, a_{3}, 0, a_{5}, 0, a_{7}, 0)
$$
such that
$$
e^{a_{1}}+e^{-a_{1}}=3, e^{a_{3}}+e^{-a_{3}}=4, e^{a_{5}}+e^{-a_{5}}=5, e^{a_{7}}+e^{-a_{7}}=6
$$
to perturb the potential function. Then the bulk-deformed potential function is
\begin{equation}
\mathfrak{PO}_{\mathfrak{b}}^{L(u)}=3y_{1}+y_{2}+4y_{1}^{-1}y_{2}+y_{1}^{-2}y_{2}+5y_{1}^{-1}+y_{2}^{-1}+6y_{1}y_{2}^{-1}+y_{1}^{2}y_{2}^{-1}
\end{equation}
which has 8 nondegenerate critical points. Therefore with respect to this bulk deformation the deformed Kodaira-Spencer map is an isomorphism and it gives us a quasi-morphism $\mu_{\mathfrak{b}}$.

\bibliographystyle{amsplain}

\end{document}